\documentclass[11pt,reqno,letterpaper]{amsart}
\usepackage{amsmath,amsthm,amsfonts,amssymb}

\newtheorem{theorem}{Theorem}
\newtheorem{lemma}[theorem]{Lemma}
\newtheorem{proposition}[theorem]{Proposition}
\newtheorem{corollary}[theorem]{Corollary}

\numberwithin{equation}{section}

\begin{document}

\title[One modification of the martingale transform]{One modification of the martingale transform
and its applications to paraproducts and stochastic integrals}
\thanks{This work has been fully supported by the Croatian Science Foundation under the project 3526.}

\author{Vjekoslav Kova\v{c}}
\address{Vjekoslav Kova\v{c}, Department of Mathematics, Faculty of Science, University of Zagreb, Bijeni\v{c}ka cesta 30, 10000 Zagreb, Croatia}
\email{vjekovac@math.hr}

\author{Kristina Ana \v{S}kreb}
\address{Kristina Ana \v{S}kreb, Faculty of Civil Engineering, University of Zagreb, Fra Andrije Ka\v{c}i\'{c}a Mio\v{s}i\'{c}a 26, 10000 Zagreb, Croatia}
\email{kskreb@grad.hr}


\begin{abstract}
In this paper we introduce a variant of Burkholder's martingale transform associated with two martingales with respect to
different filtrations.
Even though the classical martingale techniques cannot be applied, we show that the discussed transformation still satisfies some
expected $\mathrm{L}^p$ estimates.
Then we apply the obtained inequalities to general-dilation twisted paraproducts, particular instances of which
have already appeared in the literature.
As another application we construct stochastic integrals $\int_{0}^{t}H_s d(X_s Y_s)$ associated with certain
continuous-time martingales $(X_t)_{t\geq 0}$ and $(Y_t)_{t\geq 0}$.
The process $(X_t Y_t)_{t\geq 0}$ is shown to be a ``good integrator'', although it is not necessarily a semimartingale,
or even adapted to any convenient filtration.
\end{abstract}

\maketitle

\section{Introduction and statement of the results}

\subsection{Discrete-time estimates}
\label{subsecintro}
If $(U_k)_{k=0}^{\infty}$ and $(V_k)_{k=0}^{\infty}$ are two completely arbitrary discrete-time stochastic processes,
let us agree to write \hbox{$(U\cdot V)_{n=0}^{\infty}$} for a new process defined by
\begin{equation}\label{eqdotdef}
(U\cdot V)_n := \sum_{k=1}^{n} U_{k-1} (V_{k}-V_{k-1})
\end{equation}
and adopt the convention $(U\cdot V)_0=0$.
In the particular case when $(V_k)_{k=0}^{\infty}$ is a martingale and $(U_k)_{k=0}^{\infty}$
is bounded and adapted with respect to the same filtration, the above process is precisely Burkholder's \emph{martingale transform} \cite{Burk1}.
It plays an important role in finding sharp estimates for singular integral operators \cite{BW}, the theory of UMD spaces \cite{Burk3},
and inequalities for stochastic integrals \cite{Burk2}.
See \cite{Burk4} and \cite{Banu} for more details and references to the extensive literature.
However, here we consider a different setting, which is motivated by a probabilistic technique in the proof of boundedness
of a certain two-dimensional paraproduct-type operator \cite{Kov2}.

Let us begin by describing a special case of two filtrations $(\mathcal{F}_k)_{k=0}^{\infty}$ and $(\mathcal{G}_k)_{k=0}^{\infty}$
that will be used throughout this work.
Suppose that the underlying probability space is the product
$(\Omega_1\!\times\!\Omega_2,\mathcal{A}\otimes\mathcal{B},\mathbb{P}_1\!\times\!\mathbb{P}_2)$
of two probability spaces $(\Omega_1,\mathcal{A},\mathbb{P}_1)$ and $(\Omega_2,\mathcal{B},\mathbb{P}_2)$.
Whenever we write $\mathbb{E}$ alone, it will be understood that the expectation is taken with respect to the product measure
$\mathbb{P}=\mathbb{P}_1\times\mathbb{P}_2$. Similarly we do with the Lebesgue spaces and their norms.
Suppose that we are also given two filtrations $(\mathcal{A}_k)_{k=0}^{\infty}$ and $(\mathcal{B}_k)_{k=0}^{\infty}$ of
$\mathcal{A}$ and $\mathcal{B}$ respectively and denote
\begin{equation}\label{eqfiltrations}
\mathcal{F}_k := \mathcal{A}_k\otimes\mathcal{B}, \quad \mathcal{G}_k := \mathcal{A}\otimes\mathcal{B}_k
\end{equation}
for each nonnegative integer $k$.
We can think of $(\mathcal{F}_k)_{k=0}^{\infty}$ and $(\mathcal{G}_k)_{k=0}^{\infty}$ as being
a ``horizontal'' and a ``vertical'' filtration of $\mathcal{A}\otimes\mathcal{B}$ respectively.
We remark that the two filtrations in \eqref{eqfiltrations} are not necessarily independent --- in fact they rarely are.
Proposition~\ref{propcommuting} in the closing section will help us develop the intuition
by showing that sigma algebras $\mathcal{F}_k$ and $\mathcal{G}_\ell$ are indeed independent
conditionally on $\mathcal{F}_k\cap\mathcal{G}_\ell$.

Suppose that $(X_k)_{k=0}^{\infty}$ is a real-valued martingale with respect to the filtration $(\mathcal{F}_k)_{k=0}^{\infty}$ and
that $(Y_k)_{k=0}^{\infty}$ is a real-valued martingale with respect to $(\mathcal{G}_k)_{k=0}^{\infty}$.
Finally, let $(K_k)_{k=0}^{\infty}$ be an adapted process with respect to the filtration $(\mathcal{F}_k\cap\mathcal{G}_k)_{k=0}^{\infty}$.
For processes $((KX\cdot Y)_n)_{n=0}^{\infty}$ and $((K\cdot XY)_n)_{n=0}^{\infty}$ our Definition \eqref{eqdotdef} unfolds as
\begin{align*}
(KX\cdot Y)_n & = \sum_{k=1}^{n} K_{k-1} X_{k-1} (Y_{k}-Y_{k-1}) , \\
(K\cdot XY)_n & = \sum_{k=1}^{n} K_{k-1} (X_{k}Y_{k}-X_{k-1}Y_{k-1}) .
\end{align*}
These processes are no longer adapted to any convenient filtration, so they cannot be treated in the same way as Burkholder's transform.
We further discuss those difficulties in Section~\ref{secclosing}.
Nevertheless, they still prove to be useful and they still satisfy some $\mathrm{L}^p$ estimates.

Let us adopt the notation $\|U\|_{\mathrm{L}^{p}}:=(\mathbb{E}|U|^p)^{1/p}$ for any random variable $U$ and $1\leq p<\infty$,
while $\|U\|_{\mathrm{L}^{\infty}}$ is simply defined as the essential supremum of $|U|$.
For some of the desired estimates we will need that the intersection of filtrations satisfies the following ``uniform growth'' property:
there exists a constant $A$ such that
\begin{equation}\label{eqjumpscondition}
\big\|\mathbb{E}(U|\mathcal{F}_{k+1}\cap\mathcal{G}_{k+1})\big\|_{\mathrm{L}^\infty}
\leq A\,\big\|\mathbb{E}(U|\mathcal{F}_k\cap\mathcal{G}_k)\big\|_{\mathrm{L}^\infty}
\end{equation}
for any random variable $U\geq 0$ and any integer $k\geq 0$.

\begin{theorem}\label{theoremdiscrete}
\begin{itemize}
\item[(a)]
There exists an absolute constant $C$ such that for each nonnegative integer $n$ we have the inequalities:
\begin{align}
\|(KX\cdot Y)_n \|_{\mathrm{L}^{4/3}} & \leq C \,\|X_n\|_{\mathrm{L}^{4}}\|(K\cdot Y)_n\|_{\mathrm{L}^{2}},
\label{eqdiscreteestimate1} \\
\|(K\cdot XY)_n \|_{\mathrm{L}^{4/3}} & \leq C \,\big(\|X_n\|_{\mathrm{L}^{4}}\|(K\cdot Y)_n\|_{\mathrm{L}^{2}}
+\|Y_n\|_{\mathrm{L}^{4}}\|(K\cdot X)_n\|_{\mathrm{L}^{2}}\big).
\label{eqdiscreteestimate}
\end{align}
\item[(b)]
Let us additionally suppose that the filtration $(\mathcal{F}_k\cap\mathcal{G}_k)_{k=0}^{\infty}$ satisfies Condition \eqref{eqjumpscondition}.
For any exponents $p,q,r$ from the range $1/r=1/p+1/q$,\, $1<r<2<p,q<\infty$ there exists a constant $C_{p,q,r}$
such that for each nonnegative integer $n$ we have
\begin{equation}\label{eqdiscreteestimate3}
\|(KX\cdot Y)_n \|_{\mathrm{L}^{r}} \leq C_{p,q,r}\, A^{3/2}\, \big(\max_{0\leq k\leq n-1}\|K_k\|_{\mathrm{L}^\infty}\big)\,
\|X_n\|_{\mathrm{L}^{p}}\|Y_n\|_{\mathrm{L}^{q}}.
\end{equation}
\end{itemize}
\end{theorem}

We emphasize that the constants $C$ and $C_{p,q,r}$ do not depend on the filtrations or the processes involved.
We do not know if Condition \eqref{eqjumpscondition} is necessary in order to have Estimate \eqref{eqdiscreteestimate3} ---
we simply need it for the stopping time argument in our proof.
On the other hand, in Estimates \eqref{eqdiscreteestimate1} and \eqref{eqdiscreteestimate} we have avoided imposing any conditions on the martingales.

A typical example of a filtration with uniformly bounded jumps is the one generated by dyadic cubes in $\mathbb{R}^d$.
Indeed, Inequality \eqref{eqdiscreteestimate3} was already established in \cite{Kov2} for the particular case when
$(\mathcal{A}_k)_{k=0}^{\infty}$ and $(\mathcal{B}_k)_{k=0}^{\infty}$ are both just the standard one-dimensional dyadic filtration.

Let us now turn to applications of Theorem~\ref{theoremdiscrete}, which motivate our general setting.

\subsection{An application to general-dilation twisted paraproducts}
\label{subsecparaprod}
In this subsection we use the advantage of having at our disposal Estimate \eqref{eqdiscreteestimate3} for rather general martingales.

A somewhat unusual variant of paraproduct, the so-called \emph{twisted paraproduct},
was suggested by Demeter and Thiele in \cite{DT} as a particular case of the
two-dimensional variant of the well-known bilinear Hilbert transform \cite{LT1}, \cite{LT2}.
We intend to generalize the $\mathrm{L}^p$ estimate for this operator,
which was the main result of paper \cite{Kov2}, to the setting of general groups of dilations.
Even though the proof of Corollary \ref{cortwisted} below follows the same outline as the one in \cite{Kov2},
we want to present how Theorem \ref{theoremdiscrete} can be translated to convolution-type operators
using the construction by Christ \cite{Chr} of dyadic cubes in a space of homogeneous type
and the full strength of the square function estimate of Jones, Seeger, and Wright \cite{JSW}.

Let us begin by introducing a general dilation structure and we borrow much of the setting from \cite{SW} and \cite{JSW}.
Take two positive integers $d_1,d_2$ and set $d=d_1+d_2$, so that the Euclidean space $\mathbb{R}^d$ splits as $\mathbb{R}^{d_1}\times\mathbb{R}^{d_2}$.
Let $(\delta_{t}^{(1)})_{t>0}$ and $(\delta_{t}^{(2)})_{t>0}$ be two multiplicative single-parameter \emph{groups of dilations} on
$\mathbb{R}^{d_1}$ and $\mathbb{R}^{d_2}$ respectively. In other words,
\begin{itemize}
\item $\delta_{t}^{(j)}\colon\mathbb{R}^{d_j}\to\mathbb{R}^{d_j}$ is a linear transformation for $t>0$, $j=1,2$,
\item $\delta_{1}^{(j)}$ is the identity and $\delta_{st}^{(j)}=\delta_{s}^{(j)}\delta_{t}^{(j)}$ for $s,t>0$, $j=1,2$,
\item the action $\langle 0,+\infty\rangle\times\mathbb{R}^{d_j}\to\mathbb{R}^{d_j}$, \,$(t,x)\mapsto\delta_{t}^{(j)}x$ is continuous for $j=1,2$,
\item $\lim_{t\to 0}\delta_{t}^{(j)}x = \mathbf{0}$ for $x\in\mathbb{R}^{d_j}$, $j=1,2$.
\end{itemize}
It is easily seen that the dilation groups must be of the form
$$ \delta_{t}^{(j)} = t^{A_j} = e^{(\log t)A_j} , \quad t>0,\ j=1,2 $$
for some matrices $A_j\in\mathrm{M}_{d_j}(\mathbb{R})$ such that all of their eigenvalues have positive real parts.
A familiar example are \emph{non-isotropic dilations} on $\mathbb{R}^n$, defined as
$$ \delta_{t}(x_1,\ldots,x_n) := \big(t^{a_1}x_1,\ldots,t^{a_n}x_n\big) , $$
in which case the corresponding generator-matrix is $\mathop{\mathrm{diag}}(a_1,\ldots,a_n)$.

Next, take two Schwartz functions $\varphi^{(j)}\colon\mathbb{R}^{d_j}\to\mathbb{C}$, $j=1,2$ normalized by
$\int_{\mathbb{R}^{d_j}}\varphi^{(j)}(x)dx=1$.
Their dilates will be denoted
$$ \varphi_{t}^{(j)}(x):= \big(\det\delta_{t}^{(j)}\big)\,\varphi^{(j)}\big(\delta_{t}^{(j)}x\big),
\quad x\in\mathbb{R}^{d_j},\ t>0,\ j=1,2. $$
Let us discretize the scales by taking some parameters $0<\alpha,\beta<1$ and considering their integer powers.
A particular choice of $\alpha,\beta$ will come up in the proof,
similarly as it is rather canonical to take $\alpha=\beta=2$ for the standard dyadic dilation structure.
Therefore, for any integer $k$ denote by $\mathrm{P}_{k}^{(1)}, \mathrm{P}_{k}^{(2)}$ the following smooth ``projections'',
\begin{align*}
(\mathrm{P}_{k}^{(1)}f)(x,y) & := \int_{\mathbb{R}^{d_1}} f(x-u,y) \varphi_{\alpha^k}^{(1)}(u) du , \\
(\mathrm{P}_{k}^{(2)}f)(x,y) & := \int_{\mathbb{R}^{d_2}} f(x,y-v) \varphi_{\beta^k}^{(2)}(v) dv ,
\end{align*}
i.e.\@ $\mathrm{P}_{k}^{(j)}f$, $j=1,2$ are partial convolutions of $f$ with the dilates of $\varphi^{(j)}$.
In analogy with \cite{Kov2} it is natural to define a \emph{general-dilation twisted paraproduct} as
$$ T_{\alpha,\beta}(f,g) := \sum_{k\in\mathbb{Z}}\, \big(\mathrm{P}_{k}^{(1)}f\big) \,\big(\mathrm{P}_{k+1}^{(2)}g-\mathrm{P}_{k}^{(2)}g\big) . $$
This expression is well-defined for instance for compactly supported $\mathrm{C}^1$ functions $f$ and $g$.
A noteworthy feature of this paraproduct-type operator is that it possesses enough cancellation, although
the projections $\mathrm{P}_{k}^{(1)}$ and the differences $\mathrm{P}_{k+1}^{(2)}-\mathrm{P}_{k}^{(2)}$ act in separate sets of variables.
We can formulate a boundedness result in this general setting.

\begin{corollary}\label{cortwisted}
There exist parameters $0<\alpha,\beta<1$, depending only on the dilation structure, such that the estimate
\begin{equation}\label{eqtwistedest}
\big\|T_{\alpha,\beta}(f,g)\big\|_{\mathrm{L}^r(\mathbb{R}^{d})}
\leq C \|f\|_{\mathrm{L}^p(\mathbb{R}^{d})} \|g\|_{\mathrm{L}^q(\mathbb{R}^{d})}
\end{equation}
holds whenever $1<r<2<p,q<\infty$ and $1/r=1/p+1/q$, with a constant $C$
depending on $p,q,r$, $\alpha,\beta$, $\varphi^{(1)},\varphi^{(2)}$, and the dilation groups.
\end{corollary}

Corollary \ref{cortwisted} can be extended to more general singular integral operators, to a larger range of $\mathrm{L}^p$ spaces,
or even to certain Sobolev spaces, as it was done for the geometrically ``flat'' particular case
in \cite{Ber}, \cite{BK}, \cite{Kov1}, \cite{Kov2}, and \cite{KT}.
We do not discuss any of these generalizations here, as the topic diverges from the martingale method once we have established
Estimate \eqref{eqtwistedest} and because we do not have any further ideas to present in these directions.

\subsection{An application to non-adapted stochastic integrals}
\label{subsecintegral}
The motivation behind the material in this subsection lies in presenting a possible direction in which It\={o}'s integration theory \cite{Ito}
can be extended beyond the limitations of the Bichteler-Dellacherie theorem.

This time we prefer to take continuous-time filtrations $(\mathcal{F}_t)_{t\geq 0}$ and $(\mathcal{G}_t)_{t\geq 0}$,
constructed analogously as in Subsection \ref{subsecintro}.
A natural example can be obtained by taking two mutually independent stochastic processes
$(A_t)_{t\geq 0}$ and $(B_t)_{t\geq 0}$ constructed on a product space and setting
\begin{equation}\label{eqfiltrations2}
\begin{array}{l}
\mathcal{F}_{t} := \sigma\big( \big\{A_s : 0\leq s\leq t \big\}\big) \vee \sigma\big(\big\{B_s : 0\leq s<\infty \big\} \big) , \\[1.5mm]
\mathcal{G}_{t} := \sigma\big( \big\{A_s : 0\leq s<\infty \big\}\big) \vee \sigma\big(\big\{B_s : 0\leq s\leq t \big\} \big) .
\end{array}
\end{equation}
Intuitively, $(\mathcal{F}_t)_{t\geq 0}$ is progressively following the first process but
does not filter any information relevant to the second one.
A sufficiently interesting case is already obtained when $(A_t,B_t)_{t\geq 0}$ is a two-dimensional Brownian motion.

Suppose that $(X_s)_{s\geq 0}$ is a real-valued martingale with respect to $(\mathcal{F}_s)_{s\geq 0}$
and that $(Y_s)_{s\geq 0}$ is a real-valued martingale with respect to $(\mathcal{G}_s)_{s\geq 0}$.
Let us also fix $t>0$ and additionally assume that $X_t,Y_t\in\mathrm{L}^4$.
Suppose that we would like to construct the stochastic integral
\begin{equation}\label{eqstochint}
\int_{0}^{t} H_s d(X_s Y_s) ,
\end{equation}
where $(H_s)_{s\geq 0}$ is a predictable process with respect to the filtration $(\mathcal{F}_s\cap\mathcal{G}_s)_{s\geq 0}$.
We emphasize that \eqref{eqstochint} is not the usual stochastic integral, so results of the classical integration theory cannot be applied.
Indeed, the integrator $(X_s Y_s)_{s\geq 0}$ need not be adapted to any reasonable filtration
and it also does not necessarily have paths of bounded variation.
We refer to Section~\ref{secclosing} for an illustrative example.
It might be somewhat unexpected that this process is still a ``good integrator''
in the sense of Corollaries~\ref{corbichdell} and \ref{corito} below.

Let us begin by taking an elementary $(\mathcal{F}_s\cap\mathcal{G}_s)_{s\geq 0}$-predictable integrand.
It is a process $H=(H_s)_{s\geq 0}$ given explicitly by
\begin{equation}\label{eqelementaryprocess}
H_s = \left\{\begin{array}{cl} K_{-1} & \textup{ for } s=0, \\
K_{k-1} & \textup{ for } t_{k-1}<s\leq t_{k},\ \ k=1,2,\ldots,n, \\
0 & \textup{ for } s>t, \end{array}\right.
\end{equation}
where
\begin{equation}\label{eqpartition}
0=t_0<t_1<t_2<\ldots<t_{n-1}<t_n=t
\end{equation}
is a partition of $[0,t]$, $K_{-1}$ is $(\mathcal{F}_{0}\cap\mathcal{G}_{0})$-measurable,
and each $K_k$ is $(\mathcal{F}_{t_{k}}\cap\mathcal{G}_{t_{k}})$-measurable.
In this particular case we define Integral \eqref{eqstochint} directly as
$$ \int_{0}^{t} H_s d(X_s Y_s) := \sum_{k=1}^{n} K_{k-1} (X_{t_{k}}Y_{t_{k}}-X_{t_{k-1}}Y_{t_{k-1}}) . $$
Observe that this definition is independent of the representation of $(H_s)_{s\geq 0}$.

An easy consequence of Theorem \ref{theoremdiscrete} (a) will be the following result.
\begin{corollary}\label{corbichdell}
Under the above conditions the set
$$ \bigg\{ \int_{0}^{t} H_s d(X_s Y_s) \,:\, \textup{\emph{$(H_s)_{s\geq 0}$ is elementary $(\mathcal{F}_s\cap\mathcal{G}_s)_{s\geq 0}$-predictable}},
\ \sup_{0\leq s\leq t}\|H_s\|_{\mathrm{L}^\infty}\leq 1 \bigg\} $$
is bounded in $\mathrm{L}^{4/3}$ and thus also in probability.
\end{corollary}

The famous Bichteler-Dellacherie theorem \cite{Bic1}, \cite{Bic2}, \cite{Del} characterizes semimartingales $(Z_s)_{s\geq 0}$
as c\`{a}dl\`{a}g adapted processes for which the set of integrals $\int_{0}^{t}H_s dZ_s$ with respect to
bounded elementary predictable integrands is bounded in probability.
Corollary~\ref{corbichdell} might be interesting because it shows that our pointwise-product process $(X_s Y_s)_{s\geq 0}$ also shares this property,
although it is not necessarily a semimartingale or even adapted with respect to $(\mathcal{F}_s\cap\mathcal{G}_s)_{s\geq 0}$.
Moreover, there is no canonical way of decomposing it into a finite variation part and a local martingale part;
see the discussion in Section~\ref{secclosing}.

For an elementary predictable process given by \eqref{eqelementaryprocess} we also define the seminorm
\begin{equation}\label{eqseminorm}
\|H\|_{X,Y,t} := \Big(\mathbb{E}\,\sum_{k=1}^{n}K_{k-1}^{2}(X_{t_k}\!-\!X_{t_{k-1}})^2
+\mathbb{E}\,\sum_{k=1}^{n}K_{k-1}^{2}(Y_{t_k}\!-\!Y_{t_{k-1}})^2\Big)^{1/2} .
\end{equation}
This expression does not depend on the representation of $H$, because if $K_{k-2}=K_{k-1}$, then we have the identity
$$ \mathbb{E}\, K_{k-2}^{2}(X_{t_{k}}\!-\!X_{t_{k-2}})^2
= \mathbb{E}\, K_{k-2}^{2}\big((X_{t_{k-1}}\!-\!X_{t_{k-2}})^2+(X_{t_{k}}\!-\!X_{t_{k-1}})^2\big) , $$
which in turn is a consequence of
$$ \mathbb{E}\, K_{k-2}^{2}(X_{t_{k-1}}\!-\!X_{t_{k-2}})(X_{t_{k}}\!-\!X_{t_{k-1}})
= \mathbb{E} \Big( K_{k-2}^{2}(X_{t_{k-1}}\!-\!X_{t_{k-2}}) \,\mathbb{E}(X_{t_{k}}\!-\!X_{t_{k-1}} | \mathcal{F}_{t_{k-1}}) \Big) = 0 . $$

The following result will also follow directly from Theorem \ref{theoremdiscrete}.
More precisely, Estimate \eqref{eqdiscreteestimate} will be a substitute for the It\={o} isometry from the classical construction.
\begin{corollary}\label{corito}
There exists an absolute constant $C$ such that for $(\mathcal{F}_s)_{s\geq 0}$, $(\mathcal{G}_s)_{s\geq 0}$, $(X_s)_{s\geq 0}$, $(Y_s)_{s\geq 0}$
as before, for each $t\geq 0$, and for any elementary $(\mathcal{F}_s\cap\mathcal{G}_s)_{s\geq 0}$-predictable process $(H_s)_{s\geq 0}$ one has
\begin{equation}\label{eqcontest}
\Big\| \int_{0}^{t} H_s d(X_s Y_s) \Big\|_{\mathrm{L}^{4/3}}
\leq C \,\|H\|_{X,Y,t} \big(\|X_t\|_{\mathrm{L}^{4}} + \|Y_t\|_{\mathrm{L}^{4}}\big) .
\end{equation}
Consequently, if
$$ \big(H_s^{(j)}\big)_{s\geq 0}, \quad j=1,2,\ldots $$
is a Cauchy sequence of elementary $(\mathcal{F}_s\cap\mathcal{G}_s)_{s\geq 0}$-predictable processes
in the seminorm $\|\cdot\|_{X,Y,t}$, then the sequence of random variables
\begin{equation}\label{eqstochint0}
\int_{0}^{t} H_s^{(j)} d(X_s Y_s), \quad j=1,2,\ldots
\end{equation}
converges in $\mathrm{L}^{4/3}$.
\end{corollary}

The set of all elementary $(\mathcal{F}_s\cap\mathcal{G}_s)_{s\geq 0}$-predictable processes
\eqref{eqelementaryprocess} with finite seminorm \eqref{eqseminorm} is a linear space
and we can identify all $H$ and $\widetilde{H}$ such that $\|H-\widetilde{H}\|_{X,Y,t}=0$.
Then we let $\mathcal{P}_{X,Y,t}$ denote the completion of the obtained normed space with respect to $\|\cdot\|_{X,Y,t}$.
Corollary~\ref{corito} extends the definition of \eqref{eqstochint} by continuity to processes in $\mathcal{P}_{X,Y,t}$.
Integrals \eqref{eqstochint0} can be thought of as Riemann sums of \eqref{eqstochint}.
Therefore, the stochastic integral is defined in the strong sense,
as a limit in $\mathrm{L}^{4/3}$ and thus also in probability.

We have intentionally avoided any mention of quadratic variations of the two martingales in question
and so did not need to impose any conditions on the filtrations or martingale paths that would guarantee their existence.
However, if the two filtrations satisfy the ``usual hypotheses'' from \cite[I.1]{Pro}
and quadratic variations $(\langle X\rangle_s)_{s\geq 0}$ and $(\langle Y\rangle_s)_{s\geq 0}$ are available,
then Corollary~\ref{corito} actually extends integral \eqref{eqstochint} to predictable integrands $(H_s)_{s\geq 0}$ satisfying
$$ \mathbb{E}\int_{0}^{t}H_{s}^{2}\,d\big(\langle X\rangle_s\!+\!\langle Y\rangle_s\big)<\infty . $$

A rather naive interpretation of Quantity \eqref{eqstochint} is as follows.
A unit stock price is formed as a product of two martingales $(X_s)_{s\geq 0}$ and $(Y_s)_{s\geq 0}$
that ``progress'' with two independent processes $(A_s)_{s\geq 0}$ and $(B_s)_{s\geq 0}$.
However, $X_t$ depends on the amount of information from the first process available up to time $t$,
but we do not know how it depends on the second process.
This accounts to $X_t$ being measurable with respect to the $\sigma$-algebra $\mathcal{F}_t$ defined by \eqref{eqfiltrations2}.
We also impose an analogous requirement on $(Y_s)_{s\geq 0}$.
The random variable \eqref{eqstochint} equals the accumulated gain up to time $t$ when using a trading strategy $(H_s)_{s\geq 0}$.

Let us close this section with a comment that we did not insist on finding minimal conditions required
for the construction of \eqref{eqstochint}.
For instance, the two martingales could be local in the sense that a sequence of stopping times
with respect to $(\mathcal{F}_s\cap\mathcal{G}_s)_{s\geq 0}$
reduces them to $\mathrm{L}^4$ martingales, in analogy with \cite[I.6]{Pro}.
On the other hand, the partitions of $[0,t]$ could have also been random,
consisting of stopping times with respect to $(\mathcal{F}_s\cap\mathcal{G}_s)_{s\geq 0}$, as in \cite[II.5]{Pro}.
Keeping the exposition simple seems to be a useful tradeoff that
emphasizes the main novelty in discrete-time estimates of Subsection~\ref{subsecintro}.

\section{Construction of the control process}
\label{seccontrol}

Let us begin by performing a few easy reductions in Theorem \ref{theoremdiscrete}.
Inequality \eqref{eqdiscreteestimate} is an easy consequence of \eqref{eqdiscreteestimate1}.
One only has to perform the splitting
\begin{equation}
\label{eqsplitting}
(K\cdot XY)_n = (KX\cdot Y)_n + (KY\cdot X)_n + \sum_{k=1}^{n} K_{k-1}(X_{k}-X_{k-1})(Y_{k}-Y_{k-1}) ,
\end{equation}
which can be thought of as a discrete version of the integration by parts formula.
The first two terms on the right hand side of \eqref{eqsplitting} are analogous by symmetry.
The third term is handled using the Cauchy-Schwarz, H\"{o}lder, and Burkholder-Davis-Gundy inequalities \cite{BDG},
{\allowdisplaybreaks\begin{align*}
& \Big\|\sum_{k=1}^{n}K_{k-1}(X_{k}\!-\!X_{k-1})(Y_{k}\!-\!Y_{k-1})\Big\|_{\mathrm{L}^{4/3}} \\
& \leq \Big\|\Big(\sum_{k=1}^{n}K_{k-1}^2(X_{k}\!-\!X_{k-1})^2\Big)^{1/2}
\Big(\sum_{k=1}^{n}(Y_{k}\!-\!Y_{k-1})^2\Big)^{1/2}\Big\|_{\mathrm{L}^{4/3}} \\
& \leq \Big\|\Big(\sum_{k=1}^{n}K_{k-1}^2(X_{k}\!-\!X_{k-1})^2\Big)^{1/2}\Big\|_{\mathrm{L}^{2}}
\Big\|\Big(\sum_{k=1}^{n}(Y_{k}\!-\!Y_{k-1})^2\Big)^{1/2}\Big\|_{\mathrm{L}^{4}} \\
& \leq \widetilde{C}\|(K\cdot X)_n\|_{\mathrm{L}^{2}} \|Y_n\|_{\mathrm{L}^{4}} .
\end{align*}}
In the last line we also used the discrete-time It\={o} isometry,
$$ \|(K\cdot X)_n\|_{\mathrm{L}^{2}}^{2} = \mathbb{E}\sum_{k=1}^{n}K_{k-1}^2(X_{k}\!-\!X_{k-1})^2 . $$

Moreover, we claim that we do not lose generality if we assume that $K_k=1$ for each $k$ in the proof of Estimates
\eqref{eqdiscreteestimate1} and \eqref{eqdiscreteestimate3}.
Indeed, we have
$$ KX\cdot Y = X\cdot (K\cdot Y) , $$
so \eqref{eqdiscreteestimate1} is just a statement about martingales $(X_k)_{k=0}^{\infty}$ and $\big((K\cdot Y)_{k}\big)_{k=0}^{\infty}$.
The same trick also applies to Inequality \eqref{eqdiscreteestimate3}, because ordinary Burkoldher's martingale transform is known to be bounded:
$$ \|(K\cdot Y)_n\|_{\mathrm{L}^{q}} \leq C_q
\,\big(\max_{0\leq k\leq n-1}\|K_k\|_{\mathrm{L}^\infty}\big)\, \|Y_n\|_{\mathrm{L}^{q}} $$
for come constant $C_q$, whenever $1<q<\infty$.

Observe that the variables $X_k$ and $Y_k$ for $k>n$ do not appear in any of the formulae,
so if we denote $X:=X_n$, $Y:=Y_n$, we immediately reduce to the situation when
$$ \mathbb{E}(X|\mathcal{F}_k)=X_k \,\textup{ and }\, \mathbb{E}(Y|\mathcal{G}_k)=Y_k \,\textup{ for each } k . $$
At this point we start using letters $X,Y$ to denote both random variables and the martingales $(X_k)_{k=0}^{\infty},(Y_k)_{k=0}^{\infty}$.
This will not cause confusion, because the filtrations are fixed.
The above reductions are understood throughout both the current section and the following one.

Dualizing \eqref{eqdiscreteestimate1} we see that we need to show
\begin{equation}\label{eqdiscestdualized}
\big|\mathbb{E}\big((X\cdot Y)_n \,Z\big)\big| \leq C\,\|X\|_{\mathrm{L}^{4}}\|Y\|_{\mathrm{L}^{2}}\|Z\|_{\mathrm{L}^{4}} ,
\end{equation}
for $Z\in\mathrm{L}^{4}$, while \eqref{eqdiscreteestimate3} is equivalent to
\begin{equation}\label{eqdiscestdualized2}
\big|\mathbb{E}\big((X\cdot Y)_n \,Z\big)\big| \leq C_{p,q,r} A^{3/2} \|X\|_{\mathrm{L}^{p}}\|Y\|_{\mathrm{L}^{q}}\|Z\|_{\mathrm{L}^{r'}}
\end{equation}
for an arbitrary random variable $Z\in\mathrm{L}^{r'}$,
where $r'$ is the conjugated exponent of $r$, i.e.\@ $r'=r/(r-1)$.

\smallskip
Recall that
$$ \mathbb{E}\big((X\cdot Y)_n \,Z\big) = \sum_{k=0}^{n-1} \mathbb{E}\big((Y_{k+1}\!-\!Y_{k})X_{k}Z\big) . $$
By writing
$$ X_{k} Z = \big(X_{k}Z - \mathbb{E}(X_{k}Z|\mathcal{G}_{k+1})\big)
+ \big(\mathbb{E}(X_{k}Z|\mathcal{G}_{k+1}) - \mathbb{E}(X_{k}Z|\mathcal{G}_{k})\big) + \mathbb{E}(X_{k}Z|\mathcal{G}_{k}) $$
and using the martingale property of $(Y_k)_{k=0}^{\infty}$ we obtain
\begin{align*}
\mathbb{E}\big((Y_{k+1}\!-\!Y_{k}) X_{k} Z\big) &
= \mathbb{E}\bigg(\underbrace{\mathbb{E}\Big( (Y_{k+1}\!-\!Y_{k})
\big(X_{k}Z - \mathbb{E}(X_{k}Z|\mathcal{G}_{k+1})\big) \Big| \mathcal{G}_{k+1} \Big)}_{=0} \bigg) \\
& \quad + \mathbb{E}\Big((Y_{k+1}\!-\!Y_{k})
\big(\mathbb{E}(X_{k}Z|\mathcal{G}_{k+1}) - \mathbb{E}(X_{k}Z|\mathcal{G}_{k})\big) \Big) \\
& \quad + \mathbb{E}\bigg(\underbrace{\mathbb{E}\Big( (Y_{k+1}\!-\!Y_{k})
\mathbb{E}(X_{k}Z|\mathcal{G}_{k}) \Big| \mathcal{G}_{k} \Big)}_{=0} \bigg) .
\end{align*}
We insert the conditional expectation with respect to $\mathcal{F}_{k}\cap\mathcal{G}_{k}$ in the remaining term,
so that altogether we get
\begin{equation}\label{eqsplitintoalphas}
\mathbb{E}\big((X\cdot Y)_n \,Z\big) = \mathbb{E}\sum_{k=0}^{n-1} \alpha_{k}(X,Y,Z) ,
\end{equation}
where we have denoted
$$ \alpha_{k}(X,Y,Z) := \mathbb{E}\Big((Y_{k+1}\!-\!Y_{k})
\big(\mathbb{E}(X_{k}Z|\mathcal{G}_{k+1}) - \mathbb{E}(X_{k}Z|\mathcal{G}_{k})\big) \Big| \mathcal{F}_{k}\cap\mathcal{G}_{k}\Big) . $$
Thus, we choose to work with a process that is ``artificially'' adapted to $(\mathcal{F}_{k}\cap\mathcal{G}_{k})_{k=0}^{\infty}$,
even though it does not possess any typical martingale properties.

Our approach to bounding $\sum_{k=0}^{n-1}|\alpha_{k}|$ is to find an appropriate control process
in the sense of Proposition~\ref{propbellman} below.
This technique is sometimes called the method of Bellman functions,
but the name and the idea actually come from optimal control theory.
An interested reader can consult survey articles \cite{NT} and \cite{NTV}.
Our modification does not really introduce any control parameters and we only keep the idea
of constructing an auxiliary process with required ``convexity'' properties.

\begin{proposition}\label{propbellman}
There exists a process $(\beta_{k}(X,Y,Z))_{k=0}^{\infty}$
that is adapted to the filtration\linebreak \hbox{$(\mathcal{F}_k\cap\mathcal{G}_k)_{k=0}^{\infty}$} and satisfies
\begin{align}
& |\alpha_{k}(X,Y,Z)| \,\leq\, \mathbb{E}\big(\beta_{k+1}(X,Y,Z) \big|\mathcal{F}_{k}\cap\mathcal{G}_{k}\big) - \beta_{k}(X,Y,Z)
\label{eqdiscauxalpha} , \\
& 0 \,\leq\, \beta_{k}(X,Y,Z) \,\leq\, \frac{1}{2}\,\mathbb{E}(X^2|\mathcal{F}_{k}\cap\mathcal{G}_{k})^2
+ \frac{1}{2}\,\mathbb{E}(Y^2|\mathcal{F}_{k}\cap\mathcal{G}_{k}) + \frac{1}{2}\,\mathbb{E}(Z^2|\mathcal{F}_{k}\cap\mathcal{G}_{k})^2
\label{eqdiscauxbeta}
\end{align}
for each nonnegative integer $k$.
\end{proposition}

We need to introduce a bit nonstandard notation in order to be able to write down the desired process.
Operators $\mathbb{E}_{\mathcal{A}_{k}}^{\omega'\!\to\omega}$ and $\Delta_{\mathcal{A}_{k}}^{\omega'\!\to\omega}$
acting on a random variable $U$ are defined as
$$ \mathbb{E}_{\mathcal{A}_{k}}^{\omega'\!\to\omega} U(\omega') := \mathbb{E}(U|\mathcal{A}_{k})(\omega),
\qquad\Delta_{\mathcal{A}_{k}}^{\omega'\!\to\omega} U(\omega') :=
\mathbb{E}_{\mathcal{A}_{k+1}}^{\omega'\!\to\omega} U(\omega') - \mathbb{E}_{\mathcal{A}_{k}}^{\omega'\!\to\omega} U(\omega') $$
for $U\in\mathrm{L}^{1}(\Omega_1,\mathcal{A},\mathbb{P}_1)$, $\omega\in\Omega_1$, and a nonnegative integer $k$.
Note that $\mathbb{E}(U|\mathcal{A}_{k})$ is the conditional expectation with respect to a sub-$\sigma$-algebra $\mathcal{A}_k$
of the probability space $(\Omega_1,\mathcal{A},\mathbb{P}_1)$.
We will be dealing with expressions like $V(\omega',\omega'',\ldots)$ and we will need to apply these operators fiberwise,
for instance in the ``variable'' $\omega'$ only.
For that reason we emphasize notationally in $\mathbb{E}_{\mathcal{A}_{k}}^{\omega'\!\to\omega}$ and $\Delta_{\mathcal{A}_{k}}^{\omega'\!\to\omega}$
that conditional expectations are taken in $\omega'$ and the results are evaluated in $\omega$.
This will prevent confusion in the later computations.
We define $\mathbb{E}_{\mathcal{B}_{k}}^{\omega'\!\to\omega}$ and $\Delta_{\mathcal{B}_{k}}^{\omega'\!\to\omega}$ analogously.
Several formulae for manipulation with these operators are given in the following lemma.

\begin{lemma}\label{lemmaformulae}
\begin{itemize}
\item[(a)]
For any $k\geq 0$, \,$V\in\mathrm{L}^{1}(\Omega_1\!\times\!\Omega_1,\mathcal{A}\!\otimes\!\mathcal{A},\mathbb{P}_1\!\times\!\mathbb{P}_1)$,
and $\omega\in\Omega_1$ we have
\begin{equation}\label{eqformulaa}
\mathbb{E}_{\mathcal{A}_{k}}^{\omega'\!\to\omega} \mathbb{E}_{\mathcal{A}_{k}}^{\omega''\!\to\omega'} V(\omega',\omega'')
= \mathbb{E}_{\mathcal{A}_{k}}^{\omega'\!\to\omega} \mathbb{E}_{\mathcal{A}_{k}}^{\omega''\!\to\omega} V(\omega',\omega'') .
\end{equation}
\item[(b)]
For any $k,\ell\geq 0$, \,$W\in\mathrm{L}^{1}(\Omega_1\!\times\!\Omega_2,\mathcal{A}\!\otimes\!\mathcal{B},\mathbb{P}_1\!\times\!\mathbb{P}_2)$,
\,$\omega_1\in\Omega_1$, and $\omega_2\in\Omega_2$ we have
\begin{align}
& \mathbb{E}_{\mathcal{A}_{k}}^{\omega'_1\!\to\omega_1}W(\omega'_1,\omega_2) = \mathbb{E}(W|\mathcal{F}_{k})(\omega_1,\omega_2),
\quad \mathbb{E}_{\mathcal{B}_{\ell}}^{\omega'_2\!\to\omega_2}W(\omega_1,\omega'_2) = \mathbb{E}(W|\mathcal{G}_{\ell})(\omega_1,\omega_2),
\label{eqformulab1} \\
& \mathbb{E}_{\mathcal{A}_{k}}^{\omega'_1\!\to\omega_1} \mathbb{E}_{\mathcal{B}_{\ell}}^{\omega'_2\!\to\omega_2} W(\omega'_1,\omega'_2)
= \mathbb{E}_{\mathcal{B}_{\ell}}^{\omega'_2\!\to\omega_2} \mathbb{E}_{\mathcal{A}_{k}}^{\omega'_1\!\to\omega_1} W(\omega'_1,\omega'_2)
= \mathbb{E}(W|\mathcal{F}_{k}\cap\mathcal{G}_{\ell})(\omega_1,\omega_2) . \label{eqformulab2}
\end{align}
\item[(c)]
For any $k\geq 0$, \,$U_1,U_2\in\mathrm{L}^{1}(\Omega_1,\mathcal{A},\mathbb{P}_1)$, and $\omega\in\Omega_1$ we have
\begin{align}
& \mathbb{E}_{\mathcal{A}_{k}}^{\omega'\!\to\omega}
\big(\Delta_{\mathcal{A}_{k}}^{\omega''\!\to\omega'}U_1(\omega'')\big)
\big(\Delta_{\mathcal{A}_{k}}^{\omega''\!\to\omega'}U_2(\omega'')\big) \nonumber \\
& = \mathbb{E}_{\mathcal{A}_{k}}^{\omega'\!\to\omega}
\Big(\big(\mathbb{E}_{\mathcal{A}_{k+1}}^{\omega''\!\to\omega'}U_1(\omega'')\big)
\big(\mathbb{E}_{\mathcal{A}_{k+1}}^{\omega''\!\to\omega'}U_2(\omega'')\big)
- \big(\mathbb{E}_{\mathcal{A}_{k}}^{\omega''\!\to\omega'}U_1(\omega'')\big)
\big(\mathbb{E}_{\mathcal{A}_{k}}^{\omega''\!\to\omega'}U_2(\omega'')\big)\Big) . \label{eqformulac}
\end{align}
\end{itemize}
All identities are understood to hold a.s.\@ with respect to the corresponding probability measure.
Parts \emph{(a)} and \emph{(c)} also hold when $(\mathcal{A}_k)_{k=0}^{\infty}$ is replaced with $(\mathcal{B}_k)_{k=0}^{\infty}$
and $(\Omega_1,\mathcal{A},\mathbb{P}_1)$ is replaced with $(\Omega_2,\mathcal{B},\mathbb{P}_2)$.
\end{lemma}

\begin{proof}[Proof of Lemma~\ref{lemmaformulae}]
(a) We begin by verifying equation \eqref{eqformulaa} in the special case when $V=\mathbf{1}_{S_1\times S_2}$ for some $S_1,S_2\in\mathcal{A}$:
\begin{align*}
& \mathbb{E}_{\mathcal{A}_{k}}^{\omega'\!\to\omega} \mathbb{E}_{\mathcal{A}_{k}}^{\omega''\!\to\omega'}
\mathbf{1}_{S_1}(\omega') \mathbf{1}_{S_2}(\omega'')
= \mathbb{E}_{\mathcal{A}_{k}}^{\omega'\!\to\omega} \big( \mathbf{1}_{S_1}(\omega') \mathbb{E}(\mathbf{1}_{S_2}|\mathcal{A}_k)(\omega') \big) \\
& = \mathbb{E}\big( \mathbf{1}_{S_1} \mathbb{E}(\mathbf{1}_{S_2}|\mathcal{A}_k) \big| \mathcal{A}_k \big)(\omega)
= \mathbb{E}(\mathbf{1}_{S_1}|\mathcal{A}_k)(\omega) \mathbb{E}(\mathbf{1}_{S_2}|\mathcal{A}_k)(\omega)
= \mathbb{E}_{\mathcal{A}_{k}}^{\omega'\!\to\omega} \mathbb{E}_{\mathcal{A}_{k}}^{\omega''\!\to\omega}
\mathbf{1}_{S_1}(\omega') \mathbf{1}_{S_2}(\omega'') .
\end{align*}
Next, we apply Dynkin's $\pi$-$\lambda$ theorem to extend the result to $V=\mathbf{1}_{S}$, $S\in\mathcal{A}\otimes\mathcal{A}$.
Finally, we use linearity and approximate by simple functions.

(b) Equations \eqref{eqformulab1} are just Fubini's theorem in disguise.
The first equality in \eqref{eqformulab2} is trivially verified for $W=\mathbf{1}_{S_1\times S_2}$, $S_1\in\mathcal{A}$, $S_2\in\mathcal{B}$
and then standard approximation arguments follow.
After establishing it, we can rewrite it using \eqref{eqformulab1} as
\begin{equation}\label{eqexpcommute}
\mathbb{E}\big(\mathbb{E}(W|\mathcal{G}_{\ell})\big|\mathcal{F}_{k}\big)
= \mathbb{E}\big(\mathbb{E}(W|\mathcal{F}_{k})\big|\mathcal{G}_{\ell}\big) .
\end{equation}
Since both sides are $(\mathcal{F}_{k}\cap\mathcal{G}_{\ell})$-measurable, they must also
be equal to $\mathbb{E}(W|\mathcal{F}_{k}\cap\mathcal{G}_{\ell})$.

(c) Subtracting the left hand side from the right hand side gives
\begin{align*}
& \mathbb{E}_{\mathcal{A}_{k}}^{\omega'\!\to\omega}
\Big(\big(\mathbb{E}_{\mathcal{A}_{k+1}}^{\omega''\!\to\omega'}U_1(\omega'')
- \mathbb{E}_{\mathcal{A}_{k}}^{\omega''\!\to\omega'}U_1(\omega'')\big)
\big(\mathbb{E}_{\mathcal{A}_{k}}^{\omega''\!\to\omega'}U_2(\omega'')\big) \\
& \qquad\quad + \big(\mathbb{E}_{\mathcal{A}_{k}}^{\omega''\!\to\omega'}U_1(\omega'')\big)
\big(\mathbb{E}_{\mathcal{A}_{k+1}}^{\omega''\!\to\omega'}U_2(\omega'')
- \mathbb{E}_{\mathcal{A}_{k}}^{\omega''\!\to\omega'}U_2(\omega'')\big)\Big) \\
& = \underbrace{\mathbb{E}\Big( \mathbb{E}(U_1|\mathcal{A}_{k+1})\!-\!\mathbb{E}(U_1|\mathcal{A}_{k}) \Big| \mathcal{A}_k \Big)(\omega)}_{=0}
\,\mathbb{E}(U_2|\mathcal{A}_{k})(\omega) \\
& \quad + \mathbb{E}(U_1|\mathcal{A}_{k})(\omega)
\,\underbrace{\mathbb{E}\Big( \mathbb{E}(U_2|\mathcal{A}_{k+1})\!-\!\mathbb{E}(U_2|\mathcal{A}_{k}) \Big| \mathcal{A}_k \Big)(\omega)}_{=0}
\ = 0 . \qedhere
\end{align*}
\end{proof}

\begin{proof}[Proof of Proposition~\ref{propbellman}]
Let us denote
$$ \gamma_{k}(V,W)(\omega_1,\omega_2) :=
\mathbb{E}_{\mathcal{A}_{k}}^{\omega'_1\!\to\omega_1}\mathbb{E}_{\mathcal{A}_{k}}^{\omega''_1\!\to\omega_1}
\big(\mathbb{E}_{\mathcal{B}_{k}}^{\omega'_2\!\to\omega_2} V(\omega'_1,\omega'_2) W(\omega''_1,\omega'_2) \big)^2 $$
for $k\geq 0$ and for $V,W\in\mathrm{L}^{4}(\mathbb{P}_1\!\times\!\mathbb{P}_2)$.
We can define explicitly
$$ \beta_{k}(X,Y,Z) := \frac{1}{2}\mathbb{E}(Y_k^2|\mathcal{F}_{k})
+ \frac{1}{2}\gamma_{k}(X,Z) + \frac{1}{4}\gamma_{k}(X,X) + \frac{1}{4}\gamma_{k}(Z,Z) . $$

Bound \eqref{eqdiscauxbeta} is verified directly.
Using the conditional Cauchy-Schwarz inequality and \eqref{eqformulab2} we obtain
\begin{align*}
\gamma_{k}(V,W)(\omega_1,\omega_2) & \leq
\mathbb{E}_{\mathcal{A}_{k}}^{\omega'_1\!\to\omega_1}\mathbb{E}_{\mathcal{A}_{k}}^{\omega''_1\!\to\omega_1}
\big(\mathbb{E}_{\mathcal{B}_{k}}^{\omega'_2\!\to\omega_2} V(\omega'_1,\omega'_2)^2\big)
\big(\mathbb{E}_{\mathcal{B}_{k}}^{\omega'_2\!\to\omega_2} W(\omega''_1,\omega'_2)^2\big) \\
& \leq \mathbb{E}(V^2|\mathcal{F}_{k}\cap\mathcal{G}_{k})(\omega_1,\omega_2)
\,\mathbb{E}(W^2|\mathcal{F}_{k}\cap\mathcal{G}_{k})(\omega_1,\omega_2) ,
\end{align*}
so that
\begin{align*}
& \gamma_{k}(X,Z) \leq \frac{1}{2}\,\mathbb{E}(X^2|\mathcal{F}_{k}\cap\mathcal{G}_{k})^{2}
+\frac{1}{2}\,\mathbb{E}(Z^2|\mathcal{F}_{k}\cap\mathcal{G}_{k})^{2}, \\
& \gamma_{k}(X,X) \leq \mathbb{E}(X^2|\mathcal{F}_{k}\cap\mathcal{G}_{k})^{2},
\quad \gamma_{k}(Z,Z) \leq \mathbb{E}(Z^2|\mathcal{F}_{k}\cap\mathcal{G}_{k})^{2} .
\end{align*}
We also need to observe
$$ \mathbb{E}(Y_{k}^{2}|\mathcal{F}_k) = \mathbb{E}\big(\mathbb{E}(Y|\mathcal{G}_k)^2\big|\mathcal{F}_k\big)
\leq \mathbb{E}\big(\mathbb{E}(Y^2|\mathcal{G}_k)\big|\mathcal{F}_k\big) = \mathbb{E}(Y^2|\mathcal{F}_k\cap\mathcal{G}_k) . $$

The most technical part of the proof is to establish \eqref{eqdiscauxalpha}.
For that purpose we transform $\alpha_{k}(X,Y,Z)$ using formulae \eqref{eqformulab2}, \eqref{eqformulab1}, and \eqref{eqformulaa} respectively
to obtain
{\allowdisplaybreaks\begin{align}
& \alpha_{k}(X,Y,Z)(\omega_1,\omega_2) \nonumber \\
& = \mathbb{E}\Big(\big(\mathbb{E}(Y|\mathcal{G}_{k+1})-\mathbb{E}(Y|\mathcal{G}_{k})\big)
\big(\mathbb{E}(X_{k}Z|\mathcal{G}_{k+1})-\mathbb{E}(X_{k}Z|\mathcal{G}_{k})\big)
\Big| \mathcal{F}_{k}\cap\mathcal{G}_{k}\Big)(\omega_1,\omega_2) \nonumber \\
& = \mathbb{E}_{\mathcal{A}_{k}}^{\omega'_1\!\to\omega_1} \mathbb{E}_{\mathcal{B}_{k}}^{\omega'_2\!\to\omega_2}
\big(\Delta_{\mathcal{B}_{k}}^{\omega''_2\!\to\omega'_2}Y(\omega'_1,\omega''_2)\big)
\Big(\Delta_{\mathcal{B}_{k}}^{\omega''_2\!\to\omega'_2}
\big(\mathbb{E}_{\mathcal{A}_{k}}^{\omega''_1\!\to\omega'_1}X(\omega''_1,\omega''_2)\big)Z(\omega'_1,\omega''_2)\Big) \nonumber \\
& = \mathbb{E}_{\mathcal{A}_{k}}^{\omega'_1\!\to\omega_1}
\mathbb{E}_{\mathcal{A}_{k}}^{\omega''_1\!\to\omega_1} \mathbb{E}_{\mathcal{B}_{k}}^{\omega'_2\!\to\omega_2}
\big(\Delta_{\mathcal{B}_{k}}^{\omega''_2\!\to\omega'_2}Y(\omega'_1,\omega''_2)\big)
\big(\Delta_{\mathcal{B}_{k}}^{\omega''_2\!\to\omega'_2}X(\omega''_1,\omega''_2)Z(\omega'_1,\omega''_2)\big) . \label{eqalphatransf}
\end{align}}
On the other hand, note that $\beta_{k}(X,Y,Z)$ is $(\mathcal{F}_{k}\cap\mathcal{G}_{k})$-measurable
and expand the right hand side of \eqref{eqdiscauxalpha} as
\begin{align*}
& \mathbb{E}\big(\beta_{k+1}(X,Y,Z)-\beta_{k}(X,Y,Z)\big|\mathcal{F}_{k}\cap\mathcal{G}_{k}\big) \\
& = \frac{1}{2}\,\mathbb{E}\big((Y_{k+1}\!-\!Y_{k})^2\big|\mathcal{F}_{k}\cap\mathcal{G}_{k}\big)
+ \frac{1}{2}\,\mathbb{E}\big(\gamma_{k+1}(X,Z)\!-\!\gamma_{k}(X,Z)\big|\mathcal{F}_{k}\cap\mathcal{G}_{k}\big) \\
& \ \ + \frac{1}{4}\,\mathbb{E}\big(\gamma_{k+1}(X,X)\!-\!\gamma_{k}(X,X)\big|\mathcal{F}_{k}\cap\mathcal{G}_{k}\big)
+ \frac{1}{4}\,\mathbb{E}\big(\gamma_{k+1}(Z,Z)\!-\!\gamma_{k}(Z,Z)\big|\mathcal{F}_{k}\cap\mathcal{G}_{k}\big) .
\end{align*}
If we denote
\begin{align*}
\delta_{k}(V,W)(\omega_1,\omega_2)
& := \mathbb{E}\big(\gamma_{k+1}(V,W)\!-\!\gamma_{k}(V,W)\big|\mathcal{F}_{k}\cap\mathcal{G}_{k}\big)(\omega_1,\omega_2) \\
& \,= \mathbb{E}_{\mathcal{A}_{k}}^{\omega'_1\!\to\omega_1}\mathbb{E}_{\mathcal{B}_{k}}^{\omega'_2\!\to\omega_2}
\Big(\mathbb{E}_{\mathcal{A}_{k+1}}^{\omega''_1\!\to\omega'_1}\mathbb{E}_{\mathcal{A}_{k+1}}^{\omega'''_1\!\to\omega'_1}
\big(\mathbb{E}_{\mathcal{B}_{k+1}}^{\omega''_2\!\to\omega'_2} V(\omega''_1,\omega''_2) W(\omega'''_1,\omega''_2) \big)^2 \\
& \qquad\qquad\qquad\qquad -\mathbb{E}_{\mathcal{A}_{k}}^{\omega''_1\!\to\omega'_1}\mathbb{E}_{\mathcal{A}_{k}}^{\omega'''_1\!\to\omega'_1}
\big(\mathbb{E}_{\mathcal{B}_{k}}^{\omega''_2\!\to\omega'_2} V(\omega''_1,\omega''_2) W(\omega'''_1,\omega''_2) \big)^2\Big) ,
\end{align*}
then we can write
\begin{align}
& \mathbb{E}\big(\beta_{k+1}(X,Y,Z)-\beta_{k}(X,Y,Z)\big|\mathcal{F}_{k}\cap\mathcal{G}_{k}\big) \nonumber \\
& = \frac{1}{2}\,\mathbb{E}\big((Y_{k+1}\!-\!Y_k)^2 \big| \mathcal{F}_k\cap\mathcal{G}_k \big)
+ \frac{1}{2}\delta_{k}(X,Z) + \frac{1}{4}\delta_{k}(X,X) + \frac{1}{4}\delta_{k}(Z,Z) . \label{eqdiscrhs}
\end{align}

Now we turn back to $\alpha_k$ and estimate it until we arrive at the expressions above.
Begin by applying the simple inequality $|ab|\leq\frac{1}{2}a^2+\frac{1}{2}b^2$ to \eqref{eqalphatransf},
\begin{align*}
& |\alpha_{k}(X,Y,Z)(\omega_1,\omega_2)| \\
& \leq \frac{1}{2}\,\mathbb{E}_{\mathcal{A}_{k}}^{\omega'_1\!\to\omega_1}
\mathbb{E}_{\mathcal{B}_{k}}^{\omega'_2\!\to\omega_2}
\big(\Delta_{\mathcal{B}_{k}}^{\omega''_2\!\to\omega'_2}Y(\omega'_1,\omega''_2)\big)^2 \\
& \ \ + \frac{1}{2}\,\underbrace{\mathbb{E}_{\mathcal{A}_{k}}^{\omega'_1\!\to\omega_1}
\mathbb{E}_{\mathcal{A}_{k}}^{\omega''_1\!\to\omega_1} \mathbb{E}_{\mathcal{B}_{k}}^{\omega'_2\!\to\omega_2}
\big(\Delta_{\mathcal{B}_{k}}^{\omega''_2\!\to\omega'_2}X(\omega''_1,\omega''_2)Z(\omega'_1,\omega''_2)\big)^2}_{\varepsilon_{k}(X,Z)}.
\end{align*}
The first term is exactly
\,$\frac{1}{2}\mathbb{E}\big((Y_{k+1}\!-\!Y_k)^2 \big| \mathcal{F}_k\cap\mathcal{G}_k \big)(\omega_1,\omega_2)$,\,
while we denote $2$ times the second term by $\varepsilon_{k}(X,Z)(\omega_1,\omega_2)$.
Repeated applications of formulae \eqref{eqformulac} and \eqref{eqformulab2} allow us to write
{\allowdisplaybreaks\begin{align*}
& \varepsilon_{k}(X,Z)(\omega_1,\omega_2) \\
& = \mathbb{E}_{\mathcal{A}_{k}}^{\omega'_1\!\to\omega_1}
\mathbb{E}_{\mathcal{A}_{k}}^{\omega''_1\!\to\omega_1} \mathbb{E}_{\mathcal{B}_{k}}^{\omega'_2\!\to\omega_2}
\Big(\big(\mathbb{E}_{\mathcal{B}_{k+1}}^{\omega''_2\!\to\omega'_2}X(\omega''_1,\omega''_2)Z(\omega'_1,\omega''_2)\big)^2 \\
& \qquad\qquad\qquad\qquad\qquad\quad
-\big(\mathbb{E}_{\mathcal{B}_{k}}^{\omega''_2\!\to\omega'_2}X(\omega''_1,\omega''_2)Z(\omega'_1,\omega''_2)\big)^2\Big) \\
& = \mathbb{E}_{\mathcal{B}_{k}}^{\omega'_2\!\to\omega_2}
\big(\mathbb{E}_{\mathcal{B}_{k+1}}^{\omega''_2\!\to\omega'_2}\mathbb{E}_{\mathcal{B}_{k+1}}^{\omega'''_2\!\to\omega'_2}
-\mathbb{E}_{\mathcal{B}_{k}}^{\omega''_2\!\to\omega'_2}\mathbb{E}_{\mathcal{B}_{k}}^{\omega'''_2\!\to\omega'_2}\big) \\
& \qquad\qquad\  \big(\mathbb{E}_{\mathcal{A}_{k}}^{\omega''_1\!\to\omega_1}X(\omega''_1,\omega''_2)X(\omega''_1,\omega'''_2)\big)
\big(\mathbb{E}_{\mathcal{A}_{k}}^{\omega'_1\!\to\omega_1}Z(\omega'_1,\omega''_2)Z(\omega'_1,\omega'''_2)\big) \\
& = \mathbb{E}_{\mathcal{B}_{k}}^{\omega'_2\!\to\omega_2}
\big(\mathbb{E}_{\mathcal{B}_{k+1}}^{\omega''_2\!\to\omega'_2}\mathbb{E}_{\mathcal{B}_{k+1}}^{\omega'''_2\!\to\omega'_2}
-\mathbb{E}_{\mathcal{B}_{k}}^{\omega''_2\!\to\omega'_2}\mathbb{E}_{\mathcal{B}_{k}}^{\omega'''_2\!\to\omega'_2}\big)
\,\mathbb{E}_{\mathcal{A}_{k}}^{\omega'_1\!\to\omega_1} \\
& \qquad\qquad\  \big(\mathbb{E}_{\mathcal{A}_{k}}^{\omega''_1\!\to\omega'_1}X(\omega''_1,\omega''_2)X(\omega''_1,\omega'''_2)\big)
\big(\mathbb{E}_{\mathcal{A}_{k}}^{\omega'''_1\!\to\omega'_1}Z(\omega'''_1,\omega''_2)Z(\omega'''_1,\omega'''_2)\big) \\
& = -\, \mathbb{E}_{\mathcal{B}_{k}}^{\omega'_2\!\to\omega_2}
\mathbb{E}_{\mathcal{B}_{k}}^{\omega''_2\!\to\omega'_2}\mathbb{E}_{\mathcal{B}_{k}}^{\omega'''_2\!\to\omega'_2}
\,\mathbb{E}_{\mathcal{A}_{k}}^{\omega'_1\!\to\omega_1} \\
& \qquad\qquad\ \Big( \big(\mathbb{E}_{\mathcal{A}_{k}}^{\omega''_1\!\to\omega'_1}X(\omega''_1,\omega''_2)X(\omega''_1,\omega'''_2)\big)
\big(\mathbb{E}_{\mathcal{A}_{k}}^{\omega'''_1\!\to\omega'_1}Z(\omega'''_1,\omega''_2)Z(\omega'''_1,\omega'''_2)\big) \\
& \quad\, + \mathbb{E}_{\mathcal{B}_{k}}^{\omega'_2\!\to\omega_2}
\mathbb{E}_{\mathcal{B}_{k+1}}^{\omega''_2\!\to\omega'_2}\mathbb{E}_{\mathcal{B}_{k+1}}^{\omega'''_2\!\to\omega'_2}
\,\mathbb{E}_{\mathcal{A}_{k}}^{\omega'_1\!\to\omega_1} \\
& \qquad\qquad\ \Big( \big(\mathbb{E}_{\mathcal{A}_{k+1}}^{\omega''_1\!\to\omega'_1}X(\omega''_1,\omega''_2)X(\omega''_1,\omega'''_2)\big)
\big(\mathbb{E}_{\mathcal{A}_{k+1}}^{\omega'''_1\!\to\omega'_1}Z(\omega'''_1,\omega''_2)Z(\omega'''_1,\omega'''_2)\big) \\
& \qquad\qquad\ - \big(\Delta_{\mathcal{A}_{k}}^{\omega''_1\!\to\omega'_1}X(\omega''_1,\omega''_2)X(\omega''_1,\omega'''_2)\big)
\big(\Delta_{\mathcal{A}_{k}}^{\omega'''_1\!\to\omega'_1}Z(\omega'''_1,\omega''_2)Z(\omega'''_1,\omega'''_2)\big) \Big) \\
& = \mathbb{E}_{\mathcal{A}_{k}}^{\omega'_1\!\to\omega_1}\mathbb{E}_{\mathcal{B}_{k}}^{\omega'_2\!\to\omega_2}
\Big( \mathbb{E}_{\mathcal{A}_{k+1}}^{\omega''_1\!\to\omega'_1} \mathbb{E}_{\mathcal{A}_{k+1}}^{\omega'''_1\!\to\omega'_1}
\big(\mathbb{E}_{\mathcal{B}_{k+1}}^{\omega''_2\!\to\omega'_2} X(\omega''_1,\omega''_2)Z(\omega'''_1,\omega''_2)\big)^2 \\
& \quad\,\underbrace{\qquad\qquad\qquad\quad - \mathbb{E}_{\mathcal{A}_{k}}^{\omega''_1\!\to\omega'_1} \mathbb{E}_{\mathcal{A}_{k}}^{\omega'''_1\!\to\omega'_1}
\big(\mathbb{E}_{\mathcal{B}_{k}}^{\omega''_2\!\to\omega'_2} X(\omega''_1,\omega''_2)Z(\omega'''_1,\omega''_2)\big)^2 \Big)}_{\delta_{k}(X,Z)} \\
& \ \ -\, \mathbb{E}_{\mathcal{B}_{k}}^{\omega'_2\!\to\omega_2}
\mathbb{E}_{\mathcal{B}_{k+1}}^{\omega''_2\!\to\omega'_2}\mathbb{E}_{\mathcal{B}_{k+1}}^{\omega'''_2\!\to\omega'_2}
\,\mathbb{E}_{\mathcal{A}_{k}}^{\omega'_1\!\to\omega_1} \\
& \qquad\underbrace{\qquad\ \Big( \big(\Delta_{\mathcal{A}_{k}}^{\omega''_1\!\to\omega'_1}X(\omega''_1,\omega''_2)X(\omega''_1,\omega'''_2)\big)
\big(\Delta_{\mathcal{A}_{k}}^{\omega'''_1\!\to\omega'_1}Z(\omega'''_1,\omega''_2)Z(\omega'''_1,\omega'''_2)\big) \Big)}_{\zeta_{k}(X,Z)}.
\end{align*}}
The first term is precisely $\delta_{k}(X,Z)(\omega_1,\omega_2)$ and we denote the negative of the second term by
$\zeta_{k}(X,Z)(\omega_1,\omega_2)$, so that
\begin{equation}\label{eqdisclhs1}
|\alpha_{k}(X,Y,Z)| \leq \frac{1}{2}\,\mathbb{E}\big((Y_{k+1}\!-\!Y_k)^2 \big| \mathcal{F}_k\cap\mathcal{G}_k \big)
+ \frac{1}{2}\delta_{k}(X,Z) + \frac{1}{2}|\zeta_{k}(X,Z)| .
\end{equation}
Finally, we deal with $\zeta_{k}(X,Z)$ by estimating it as
\begin{equation}\label{eqdiscszetaeta}
|\zeta_{k}(X,Z)| \leq \frac{1}{2}\eta_{k}(X) + \frac{1}{2}\eta_{k}(Z) ,
\end{equation}
where
$$ \eta_{k}(X)(\omega_1,\omega_2) := \mathbb{E}_{\mathcal{B}_{k}}^{\omega'_2\!\to\omega_2}
\mathbb{E}_{\mathcal{B}_{k+1}}^{\omega''_2\!\to\omega'_2}\mathbb{E}_{\mathcal{B}_{k+1}}^{\omega'''_2\!\to\omega'_2}
\,\mathbb{E}_{\mathcal{A}_{k}}^{\omega'_1\!\to\omega_1}
\big(\Delta_{\mathcal{A}_{k}}^{\omega''_1\!\to\omega'_1}X(\omega''_1,\omega''_2)X(\omega''_1,\omega'''_2)\big)^2 . $$
We transform $\eta_{k}(X)$ using \eqref{eqformulac} twice, similarly as we did before with $\varepsilon_{k}(X,Z)$,
{\allowdisplaybreaks\begin{align*}
& \eta_{k}(X)(\omega_1,\omega_2) \\
& = \mathbb{E}_{\mathcal{B}_{k}}^{\omega'_2\!\to\omega_2}
\mathbb{E}_{\mathcal{B}_{k+1}}^{\omega''_2\!\to\omega'_2}\mathbb{E}_{\mathcal{B}_{k+1}}^{\omega'''_2\!\to\omega'_2}
\,\mathbb{E}_{\mathcal{A}_{k}}^{\omega'_1\!\to\omega_1} \\
& \qquad\Big( \big(\mathbb{E}_{\mathcal{A}_{k+1}}^{\omega''_1\!\to\omega'_1}X(\omega''_1,\omega''_2)X(\omega''_1,\omega'''_2)\big)^2
- \big(\mathbb{E}_{\mathcal{A}_{k}}^{\omega''_1\!\to\omega'_1}X(\omega''_1,\omega''_2)X(\omega''_1,\omega'''_2)\big)^2 \Big) \\
& = \mathbb{E}_{\mathcal{A}_{k}}^{\omega'_1\!\to\omega_1}
\big(\mathbb{E}_{\mathcal{A}_{k+1}}^{\omega''_1\!\to\omega'_1}\mathbb{E}_{\mathcal{A}_{k+1}}^{\omega'''_1\!\to\omega'_1}
-\mathbb{E}_{\mathcal{A}_{k}}^{\omega''_1\!\to\omega'_1}\mathbb{E}_{\mathcal{A}_{k}}^{\omega'''_1\!\to\omega'_1}\big)
\,\mathbb{E}_{\mathcal{B}_{k}}^{\omega'_2\!\to\omega_2}
\big(\mathbb{E}_{\mathcal{B}_{k+1}}^{\omega''_2\!\to\omega'_2}X(\omega''_1,\omega''_2)X(\omega'''_1,\omega''_2)\big)^2 \\
& = \mathbb{E}_{\mathcal{A}_{k}}^{\omega'_1\!\to\omega_1}
\mathbb{E}_{\mathcal{A}_{k+1}}^{\omega''_1\!\to\omega'_1}\mathbb{E}_{\mathcal{A}_{k+1}}^{\omega'''_1\!\to\omega'_1}
\mathbb{E}_{\mathcal{B}_{k}}^{\omega'_2\!\to\omega_2}
\big(\mathbb{E}_{\mathcal{B}_{k+1}}^{\omega''_2\!\to\omega'_2}X(\omega''_1,\omega''_2)X(\omega'''_1,\omega''_2)\big)^2 \\
& \ \ - \mathbb{E}_{\mathcal{A}_{k}}^{\omega'_1\!\to\omega_1}
\mathbb{E}_{\mathcal{A}_{k}}^{\omega''_1\!\to\omega'_1}\mathbb{E}_{\mathcal{A}_{k}}^{\omega'''_1\!\to\omega'_1}
\mathbb{E}_{\mathcal{B}_{k}}^{\omega'_2\!\to\omega_2} \\
& \qquad\Big(\big(\mathbb{E}_{\mathcal{B}_{k}}^{\omega''_2\!\to\omega'_2}X(\omega''_1,\omega''_2)X(\omega'''_1,\omega''_2)\big)^2
+ \big(\Delta_{\mathcal{B}_{k}}^{\omega''_2\!\to\omega'_2}X(\omega''_1,\omega''_2)X(\omega'''_1,\omega''_2)\big)^2\Big) \\
& = \delta_{k}(X,X) \,-\,
\underbrace{\mathbb{E}_{\mathcal{A}_{k}}^{\omega''_1\!\to\omega_1}\mathbb{E}_{\mathcal{A}_{k}}^{\omega'''_1\!\to\omega_1}
\mathbb{E}_{\mathcal{B}_{k}}^{\omega'_2\!\to\omega_2}
\big(\Delta_{\mathcal{B}_{k}}^{\omega''_2\!\to\omega'_2}X(\omega''_1,\omega''_2)X(\omega'''_1,\omega''_2)\big)^2}_{\geq 0}.
\end{align*}}
Therefore
\begin{equation}\label{eqdisclhs2}
\eta_{k}(X) \leq \delta_{k}(X,X), \quad \eta_{k}(Z) \leq \delta_{k}(Z,Z) .
\end{equation}
Combining \eqref{eqdiscrhs}--\eqref{eqdisclhs2} we complete the proof of \eqref{eqdiscauxalpha}.
\end{proof}

\section{Proof of Theorem 1}
\label{secdiscrete}

Now we are ready to establish inequalities \eqref{eqdiscestdualized} and \eqref{eqdiscestdualized2}, which in turn will establish
Theorem \ref{theoremdiscrete}, as we have already observed.
We will find it convenient to write
$$ \mathcal{H}_k := \mathcal{F}_k\cap\mathcal{G}_k , $$
since the filtration $(\mathcal{H}_k)_{k=0}^{\infty}$ plays a prominent role in our proof.
We can certainly assume that none of the variables $X,Y,Z$ are constantly zero.

By taking expectation of \eqref{eqdiscauxalpha}, summing in $k$, and telescoping we obtain
$$ \sum_{k=0}^{n-1}\mathbb{E}|\alpha_{k}(X,Y,Z)| \leq \mathbb{E}\beta_{n}(X,Y,Z) - \mathbb{E}\beta_{0}(X,Y,Z)
\leq \mathbb{E}\beta_{n}(X,Y,Z) . $$
Consequently, by \eqref{eqsplitintoalphas},
$$ \big|\mathbb{E}\big((X\cdot Y)_n \,Z\big)\big| \leq \mathbb{E}\beta_{n}(X,Y,Z) . $$
Then by taking expectation of \eqref{eqdiscauxbeta} we get with an aid of conditional Jensen's inequality
\begin{align*}
\mathbb{E}\beta_{n}(X,Y,Z) & \leq \frac{1}{2}\,\mathbb{E}
\Big(\mathbb{E}(X^4|\mathcal{H}_n) + \mathbb{E}(Y^2|\mathcal{H}_n) + \mathbb{E}(Z^4|\mathcal{H}_n)\Big) \\
& = \frac{1}{2}\|X\|_{\mathrm{L}^4}^4 + \frac{1}{2}\|Y\|_{\mathrm{L}^2}^2 + \frac{1}{2}\|Z\|_{\mathrm{L}^4}^4 .
\end{align*}
It remains to use homogeneity of $\mathbb{E}((X\cdot Y)_n Z)$ and replace $X,Y,Z$ by
$\frac{X}{\|X\|_{\mathrm{L}^4}}, \frac{Y}{\|Y\|_{\mathrm{L}^2}}, \frac{Z}{\|Z\|_{\mathrm{L}^4}}$, so that
$$ \frac{\big|\mathbb{E}\big((X\cdot Y)_n \,Z\big)\big|}{\|X\|_{\mathrm{L}^4}\|Y\|_{\mathrm{L}^2}\|Z\|_{\mathrm{L}^4}}\leq\frac{3}{2} . $$
This finally establishes \eqref{eqdiscestdualized} and hence also \eqref{eqdiscreteestimate1}.

The proof of \eqref{eqdiscestdualized2} will be reduced by Proposition \ref{propbellman}
to a slightly more complicated but still standard stopping time argument, which we adapt from \cite{Thi} or \cite{Kov2}.
Let us fix two stopping times $\sigma$ and $\tau$ with respect to the filtration $(\mathcal{H}_k)_{k=0}^{\infty}$ such that $\sigma\leq\tau\leq n$.
Splitting with respect to all possible values of $\sigma,\tau$ by repeated applications of \eqref{eqdiscauxalpha} we obtain
$$ \mathbb{E}\Big(\sum_{k\in[\sigma,\tau)}|\alpha_{k}(X,Y,Z)|\Big|\mathcal{H}_{\sigma}\Big)
\leq \mathbb{E}\big(\beta_{\tau}(X,Y,Z)\big|\mathcal{H}_{\sigma}\big) - \beta_{\sigma}(X,Y,Z) . $$
Since $\beta_{\sigma}(X,Y,Z)\geq 0$ and both sides vanish outside
$$ \{\sigma<\tau\} \subseteq \{\sigma<n\}\cap\{\tau>0\} , $$
taking expectations we actually get
\begin{equation}\label{eqstoptimesineq0}
\mathbb{E}\sum_{k\in[\sigma,\tau)}|\alpha_{k}(X,Y,Z)| \,\leq\,
\big\|\beta_{\tau}(X,Y,Z)\mathbf{1}_{\{\tau>0\}}\big\|_{\mathrm{L}^{\infty}} \,\mathbb{P}(\sigma<n) .
\end{equation}

Because of Proposition \ref{propbellman} it is natural to
introduce three martingales with respect to $(\mathcal{H}_{k})_{k=0}^{\infty}$ defined by
$$ \mathcal{X}_k := \mathbb{E}(X^2|\mathcal{H}_{k}),
\quad \mathcal{Y}_k := \mathbb{E}(Y^2|\mathcal{H}_{k}),
\quad \mathcal{Z}_k := \mathbb{E}(Z^2|\mathcal{H}_{k}) $$
and the corresponding maximal processes
$$ \bar{\mathcal{X}}_k := \max_{0\leq \ell\leq k}\mathcal{X}_\ell,
\quad \bar{\mathcal{Y}}_k := \max_{0\leq \ell\leq k}\mathcal{Y}_\ell,
\quad \bar{\mathcal{Z}}_k := \max_{0\leq \ell\leq k}\mathcal{Z}_\ell . $$
Using \eqref{eqdiscauxbeta} Inequality \eqref{eqstoptimesineq0} now becomes
$$ \mathbb{E} \sum_{k\in[\sigma,\tau)} |\alpha_{k}(X,Y,Z)| \,\leq\, \frac{1}{2}
\,\Big(\big\|\mathcal{X}_\tau\mathbf{1}_{\{\tau>0\}}\big\|_{\mathrm{L}^\infty}^2
+ \big\|\mathcal{Y}_\tau\mathbf{1}_{\{\tau>0\}}\big\|_{\mathrm{L}^\infty}
+ \big\|\mathcal{Z}_\tau\mathbf{1}_{\{\tau>0\}}\big\|_{\mathrm{L}^\infty}^2\Big) \,\mathbb{P}(\sigma<n) . $$
We can actually establish a seemingly stronger inequality
\begin{equation}\label{eqstoptimesineq}
\mathbb{E} \sum_{k\in[\sigma,\tau)} |\alpha_{k}(X,Y,Z)| \,\leq\, \frac{3}{2}
\,\big\|\mathcal{X}_\tau\mathbf{1}_{\{\tau>0\}}\big\|_{\mathrm{L}^\infty}^{1/2}
\big\|\mathcal{Y}_\tau\mathbf{1}_{\{\tau>0\}}\big\|_{\mathrm{L}^\infty}^{1/2}
\big\|\mathcal{Z}_\tau\mathbf{1}_{\{\tau>0\}}\big\|_{\mathrm{L}^\infty}^{1/2} \,\mathbb{P}(\sigma<n) ,
\end{equation}
first under the normalization
$$ \big\|\mathcal{X}_\tau\mathbf{1}_{\{\tau>0\}}\big\|_{\mathrm{L}^\infty}
= \big\|\mathcal{Y}_\tau\mathbf{1}_{\{\tau>0\}}\big\|_{\mathrm{L}^\infty}
= \big\|\mathcal{Z}_\tau\mathbf{1}_{\{\tau>0\}}\big\|_{\mathrm{L}^\infty}=1 $$
and then in the general case, by homogeneity of $\alpha_{k}(X,Y,Z)$ in each of the variables $X,Y,Z$.

For each $m\in\mathbb{Z}$ we introduce a stopping time $T^{\mathcal{X}}_{m}$ with respect to $(\mathcal{H}_{k})_{k=0}^{\infty}$ by
$$ T^{\mathcal{X}}_{m} := \inf\{k\geq 0 : \mathcal{X}_k\geq 2^{2m}\} \wedge n $$
and define $T^{\mathcal{Y}}_{m}$ and $T^{\mathcal{Z}}_{m}$ analogously.
It is easy to observe that $T^{\mathcal{X}}_{m-1}=T^{\mathcal{X}}_{m}$ for all but finitely many $m$.
Indeed, if $2^{m}\leq\mathcal{X}_0^{1/2}$, then $T^{\mathcal{X}}_{m}=0$, while
if $2^{m}>\|X\|_{\mathrm{L}^\infty}$, then $T^{\mathcal{X}}_{m}=n$.
Consequently, we have only finitely many nonempty random intervals of the form
\begin{equation}\label{eqrandomintervals}
[T^{\mathcal{X}}_{m_1-1},T^{\mathcal{X}}_{m_1}) \cap
[T^{\mathcal{Y}}_{m_2-1},T^{\mathcal{Y}}_{m_2}) \cap
[T^{\mathcal{Z}}_{m_3-1},T^{\mathcal{Z}}_{m_3})
\end{equation}
for $m_1,m_2,m_3\in\mathbb{Z}$ and they constitute a random partition of $\{0,1,2,\ldots,n-1\}$.

Let us apply \eqref{eqstoptimesineq} to each random interval $[\sigma,\tau)$ of the form \eqref{eqrandomintervals}, i.e.\@ when
$$ \sigma = T^{\mathcal{X}}_{m_1-1} \vee T^{\mathcal{Y}}_{m_2-1} \vee T^{\mathcal{Z}}_{m_3-1} ,
\quad \tau = T^{\mathcal{X}}_{m_1} \wedge T^{\mathcal{Y}}_{m_2} \wedge T^{\mathcal{Z}}_{m_3} . $$
Observe that on the set $\{\tau=k\}$, $k\geq 1$, Condition \eqref{eqjumpscondition} implies
$$ \mathcal{X}_\tau = \mathcal{X}_k \leq A \,\|\mathcal{X}_{k-1}\|_{\mathrm{L}^{\infty}} \leq A\, 2^{2m_1} \ \textup{ a.s.} , $$
so that
$$ \big\|\mathcal{X}_\tau\mathbf{1}_{\{\tau>0\}}\big\|_{\mathrm{L}^\infty}\leq 2^{2m_1} A, \quad
\big\|\mathcal{Y}_\tau\mathbf{1}_{\{\tau>0\}}\big\|_{\mathrm{L}^\infty}\leq 2^{2m_2} A, \quad
\big\|\mathcal{Z}_\tau\mathbf{1}_{\{\tau>0\}}\big\|_{\mathrm{L}^\infty}\leq 2^{2m_3} A . $$
Also note that
\begin{align*}
\{\sigma < n\}
& = \big\{T^{\mathcal{X}}_{m_1-1}<n\big\}\cap\big\{T^{\mathcal{Y}}_{m_2-1}<n\big\}\cap\big\{T^{\mathcal{Z}}_{m_3-1}<n\big\} \\
& \subseteq \big\{\bar{\mathcal{X}}_{n}\geq2^{2m_1-2}\big\}\cap\big\{\bar{\mathcal{Y}}_{n}\geq2^{2m_2-2}\big\}\cap
\big\{\bar{\mathcal{Z}}_{n}\geq2^{2m_3-2}\big\} .
\end{align*}
Summing over all $m_1,m_2,m_3$ and using \eqref{eqstoptimesineq} we obtain
\begin{align}
\mathbb{E} \sum_{k=0}^{n-1} |\alpha_{k}(X,Y,Z)| \,\leq \
& \frac{3}{2}\, A^{3/2} \!\!\sum_{m_1,m_2,m_3\in\mathbb{Z}} 2^{m_1+m_2+m_3} \nonumber \\
& \min\big\{\mathbb{P}\big(\bar{\mathcal{X}}_{n}\geq2^{2m_1-2}\big),
\mathbb{P}\big(\bar{\mathcal{Y}}_{n}\geq2^{2m_2-2}\big),
\mathbb{P}\big(\bar{\mathcal{Z}}_{n}\geq2^{2m_3-2}\big)\big\} . \label{eqbigsum}
\end{align}
This time we decide to normalize
$$ \|X\|_{\mathrm{L}^{p}} = \|Y\|_{\mathrm{L}^{q}} = \|Z\|_{\mathrm{L}^{r'}} = 1. $$
Recall that $2<p,q,r'<\infty$, so Doob's inequality gives
$$ \sum_{m\in\mathbb{Z}} 2^{mp} \,\mathbb{P}(\bar{\mathcal{X}}_{n}\geq2^{2m})
\leq C'_p \mathbb{E}|\bar{\mathcal{X}}_{n}|^{p/2} \leq C_p \mathbb{E}|\mathcal{X}_{n}|^{p/2} \leq C_p \|X\|_{\mathrm{L}^{p}}^{p} = C_p $$
and similarly
$$ \sum_{m\in\mathbb{Z}} 2^{mq} \,\mathbb{P}(\bar{\mathcal{Y}}_{n}\geq2^{2m}) \leq C_q ,
\quad \sum_{m\in\mathbb{Z}} 2^{mr'} \mathbb{P}(\bar{\mathcal{Z}}_{n}\geq2^{2m}) \leq C_r $$
for some constants $C_p,C_q,C_r$ depending only on the exponents.
In order to control the right hand side in \eqref{eqbigsum} we split the summation range $\mathbb{Z}^3$ into
three subsets, depending on which of the numbers $p m_1,q m_2,r'm_3$ is the largest.
By symmetry it is enough to bound one of the sub-sums.
\begin{align*}
& \sum_{\substack{m_1,m_2,m_3\in\mathbb{Z}\\ p m_1\geq q m_2,\ p m_1\geq r' m_3}} 2^{m_1+m_2+m_3}
\min\big\{\mathbb{P}\big(\bar{\mathcal{X}}_{n}\geq2^{2m_1-2}\big),
\mathbb{P}\big(\bar{\mathcal{Y}}_{n}\geq2^{2m_2-2}\big),
\mathbb{P}\big(\bar{\mathcal{Z}}_{n}\geq2^{2m_3-2}\big)\big\} \\
& \leq \sum_{m_1\in\mathbb{Z}} 2^{p m_1}
\Big(\sum_{\substack{m_2\in\mathbb{Z}\\ m_2\leq (p/q)m_1}} 2^{m_2-(p/q)m_1}\Big)
\Big(\sum_{\substack{m_3\in\mathbb{Z}\\ m_3\leq (p/r')m_1}} 2^{m_3-(p/r')m_1}\Big)
\,\mathbb{P}\big(\bar{\mathcal{X}}_{n}\geq2^{2m_1-2}\big) \\
& \leq 4 \sum_{m_1\in\mathbb{Z}} 2^{p m_1} \mathbb{P}\big(\bar{\mathcal{X}}_{n}\geq2^{2m_1-2}\big)
\leq 2^{p+2}\, C_p .
\end{align*}
Using homogeneity of $\mathbb{E}((X\cdot Y)_n Z)$ once again we complete the proof of \eqref{eqdiscestdualized2}.

\section{Proof of Corollary 2}

Under the hypotheses of Subsection \ref{subsecparaprod}, Stein and Wainger \cite[II.1]{SW} constructed quasinorms
$\rho^{(j)}\colon\mathbb{R}^{d_j}\to [0,\infty)$, $j=1,2$ compatible with dilations, i.e.
$$ \rho^{(j)}(\delta_{t}^{(j)}x) = t \,\rho^{(j)}(x) \textup{ for } x\in\mathbb{R}^d,\ t>0,\ j=1,2 , $$
and such that $\mathbb{R}^{d_j}$ equipped with the $d_j$-dimensional Lebesgue measure $|\cdot|$ and the quasimetric coming from $\rho^{(j)}$
turns into a space of homogeneous type.
By this notion we understand that the Lebesgue measure is finite on $\rho^{(j)}$-balls and possesses the doubling property,
$$ \big|\mathrm{B}_{\rho^{(j)}}(x,2r)| \leq M \,|\mathrm{B}_{\rho^{(j)}}(x,r)\big| $$
for $x\in\mathbb{R}^{d_j}$, $r>0$, $j=1,2$, with some absolute constant $M$.
Therefore, we can use the construction due to Christ \cite{Chr} of the so-called \emph{dyadic cubes},
which works in the setting of a general space of homogeneous type.
There exists two collections
$$ \big\{Q_{k,i}^{(j)} \,:\, k\in\mathbb{Z},\ i\in I_k^{(j)}\big\},\ \ j=1,2,\ \
I_k^{(j)}\textup{ are countable sets of indices},  $$
of $\rho^{(j)}$-open sets $Q_{k,i}^{(j)}$ and constants $0<\gamma_1,\gamma_2<1$, \,$\varepsilon>0$, \,$M'<\infty$ with the following properties.
\begin{itemize}
\item For fixed $k\in\mathbb{Z}$ and $j\in\{1,2\}$ the sets $\{Q_{k,i}^{(j)} \,:\, i\in I_k^{(j)}\}$
form a countable partition of $\mathbb{R}^{d_j}$ up to sets of measure zero.
\item For any $k,k'\in\mathbb{Z}$, $k>k'$, $i\in I_k^{(j)}$, $i'\in I_{k'}^{(j)}$, $j\in\{1,2\}$
either $Q_{k,i}^{(j)}\subseteq Q_{k',i'}^{(j)}$\linebreak or $Q_{k,i}^{(j)}\cap Q_{k',i'}^{(j)}=\emptyset$.
\item For any $k,k'\in\mathbb{Z}$, $k>k'$, $i\in I_k^{(j)}$, $j\in\{1,2\}$
there is a unique $i'\in I_{k'}^{(j)}$ such that $Q_{k,i}^{(j)}\subseteq Q_{k',i'}^{(j)}$.
\item For any $k\in\mathbb{Z}$, $i\in I_k^{(j)}$, $j\in\{1,2\}$ there is a point $x_{k,i}^{(j)}\in\mathbb{R}^{d_j}$ such that
$$ \mathrm{B}_{\rho^{(j)}}\big(x_{k,i}^{(j)},\gamma_{j}^{k}\big)\subseteq Q_{k,i}^{(j)}\subseteq
\mathrm{B}_{\rho^{(j)}}\big(x_{k,i}^{(j)},M'\gamma_{j}^{k}\big) . $$
\item If $k\in\mathbb{Z}$, $i\in I_k^{(j)}$, $i'\in I_{k+1}^{(j)}$, $j\in\{1,2\}$ are such that $Q_{k+1,i'}^{(j)}\subseteq Q_{k,i}^{(j)}$, then
$$ \big|Q_{k+1,i'}^{(j)}\big|\geq \varepsilon\big|Q_{k,i}^{(j)}\big| . $$
\item For $k\in\mathbb{Z}$, $i\in I_k^{(j)}$, $j\in\{1,2\}$, and $\vartheta>0$ one has
$$ \big|\big\{x\in Q_{k,i}^{(j)} \,:\, \rho^{(j)}\big(x,\mathbb{R}^{d_j}\!\setminus\! Q_{k,i}^{(j)}\big)<\vartheta\gamma_{j}^{k}\big\}\big|
\leq M'\vartheta^\varepsilon \big|Q_{k,i}^{(j)}\big| . $$
\end{itemize}
We will choose $\alpha=\gamma_1$, $\beta=\gamma_2$.

For each $k\in\mathbb{Z}$ one can consider $\sigma$-algebras
$$ \mathcal{A}_k := \sigma\big(\big\{Q_{k,i}^{(1)} \,:\, i\in I_k^{(1)}\big\}\big), \quad
\mathcal{B}_k := \sigma\big(\big\{Q_{k,i}^{(2)} \,:\, i\in I_k^{(2)}\big\}\big) $$
on $\mathbb{R}^{d_1}$, $\mathbb{R}^{d_2}$ respectively
and then let $\mathcal{F}_k$ and $\mathcal{G}_k$ be $\sigma$-algebras on $\mathbb{R}^d$ defined by \eqref{eqfiltrations}.
Observe that the Lebesgue measure on $\mathbb{R}^d$ is not finite, but if restrict our attention to a single large ``product cube''
$Q_{k,i}^{(1)}\times Q_{k,i'}^{(2)}$, then we can certainly normalize the measure to obtain a probability space.
Conditional expectations with respect to $\mathcal{F}_k$ and $\mathcal{G}_k$ are simply
\begin{align*}
\mathbb{E}(f|\mathcal{F}_k)(x,y) & = \big|Q_{k,i_1}^{(1)}\big|^{-1}\!\int_{Q_{k,i_1}^{(1)}} f(u,y) du , \\
\mathbb{E}(f|\mathcal{G}_k)(x,y) & = \big|Q_{k,i_2}^{(2)}\big|^{-1}\!\int_{Q_{k,i_2}^{(2)}} f(x,v) dv ,
\end{align*}
where $i_1\in I_k^{(1)}$, $i_2\in I_k^{(2)}$ are the a.e.-unique indices such that $x\in Q_{k,i_1}^{(1)}$, $y\in Q_{k,i_2}^{(2)}$.
These filtrations $(\mathcal{F}_k)$, $(\mathcal{G}_k)$ satisfy Condition \eqref{eqjumpscondition} because for $f\geq 0$ and
$$ (x,y) \in Q_{k+1,i'_1}^{(1)}\!\times\! Q_{k+1,i'_2}^{(2)} \subseteq Q_{k,i_1}^{(1)}\!\times\! Q_{k,i_2}^{(2)} $$
we have
\begin{align*}
& \mathbb{E}(f|\mathcal{F}_{k+1}\cap\mathcal{G}_{k+1})(x,y)
\,=\, \big|Q_{k+1,i'_1}^{(1)}\big|^{-1} \big|Q_{k+1,i'_2}^{(2)}\big|^{-1}\!
\int_{Q_{k+1,i'_1}^{(1)}\!\times\! Q_{k+1,i'_2}^{(2)}} f(u,v) \,du dv \\
& \leq \big(\varepsilon\big|Q_{k,i_1}^{(1)}\big|\big)^{-1} \big(\varepsilon\big|Q_{k,i_2}^{(2)}\big|\big)^{-1}\!
\int_{Q_{k,i_1}^{(1)}\!\times\! Q_{k,i_2}^{(2)}} f(u,v) \,du dv
\,=\, \varepsilon^{-2} \,\mathbb{E}(f|\mathcal{F}_{k}\cap\mathcal{G}_{k})(x,y) ,
\end{align*}
so we can take $A=\varepsilon^{-2}$.

Fix the exponents $p,q,r$ as in the statement of Corollary \ref{cortwisted} and also take
$f\in\mathrm{L}^{p}(\mathbb{R}^d)$, $g\in\mathrm{L}^{q}(\mathbb{R}^d)$.
In order to prove Inequality \eqref{eqtwistedest} for $T_{\alpha,\beta}$, it is enough to establish the same estimate for the partial sums
$$ T_{n}(f,g) := \sum_{k=0}^{n-1} \big(\mathrm{P}_{k}^{(1)}f\big) \,\big(\mathrm{P}_{k+1}^{(2)}g-\mathrm{P}_{k}^{(2)}g\big) $$
with a constant independent of $n$.
Then we can ``shift'' the scales and finally extend from a finite range of indices $k$ to the whole $\mathbb{Z}$ by a limiting argument.
Theorem \ref{theoremdiscrete} (b) can be applied to
$$ X_k = \mathbb{E}(f|\mathcal{F}_k),\ \ Y_k = \mathbb{E}(g|\mathcal{G}_k),\ \ K_k=1 $$
and it yields
\begin{equation}\label{eqmartinparaprod}
\big\|\widetilde{T}_{n}(f,g)\big\|_{\mathrm{L}^r(\mathbb{R}^{d})}
\leq C_{p,q,r} A^{3/2} \|f\|_{\mathrm{L}^p(\mathbb{R}^{d})} \|g\|_{\mathrm{L}^q(\mathbb{R}^{d})}
\end{equation}
for the operator
$$ \widetilde{T}_{n}(f,g) := \sum_{k=0}^{n-1} \mathbb{E}(f|\mathcal{F}_k) \big(\mathbb{E}(g|\mathcal{G}_{k+1})-\mathbb{E}(g|\mathcal{G}_k)\big) $$
More precisely, one has to apply \eqref{eqdiscreteestimate3} to $(X_k)$ and $(Y_k)$ localized and normalized to a single ``large'' product cube,
because we were working in a probability space in the previous sections.
The estimate extends to the whole $\mathbb{R}^d$ due to the ``proper'' scaling $1/r = 1/p + 1/q$.
These transference arguments are standard.

Once we have Inequality \eqref{eqmartinparaprod}, it is enough to bound the difference $T_{n}(f,g)-\widetilde{T}_{n}(f,g)$.
In this last step we use the square function estimate of Jones, Seeger, and Wright \cite[\S 4, pp. 6725]{JSW},
which controls the difference between convolutions and conditional expectations.
In our setting and notation, the result from \cite{JSW} reads
\begin{equation}\label{eqjswineq}
\|\mathcal{S}_{\textup{JSW}}^{(1)}f\|_{\mathrm{L}^p(\mathbb{R}^d)} \leq C_p \|f\|_{\mathrm{L}^p(\mathbb{R}^d)} ,
\quad \|\mathcal{S}_{\textup{JSW}}^{(2)}g\|_{\mathrm{L}^q(\mathbb{R}^d)} \leq C_q \|g\|_{\mathrm{L}^q(\mathbb{R}^d)} ,
\end{equation}
where
$$ \mathcal{S}_{\textup{JSW}}^{(1)}f := \Big(\sum_{k\in\mathbb{Z}} \big|\mathrm{P}_{k}^{(1)}f - \mathbb{E}(f|\mathcal{F}_{k})\big|^2\Big)^{1/2} ,
\quad \mathcal{S}_{\textup{JSW}}^{(2)}g := \Big(\sum_{k\in\mathbb{Z}} \big|\mathrm{P}_{k}^{(2)}g - \mathbb{E}(g|\mathcal{G}_{k})\big|^2\Big)^{1/2} . $$
Clearly,
\begin{align*}
T_{n}(f,g)-\widetilde{T}_{n}(f,g) =\,
& \sum_{k=0}^{n-1} \big(\mathrm{P}_{k}^{(1)}f-\mathbb{E}(f|\mathcal{F}_k)\big)
\,\big(\mathrm{P}_{k+1}^{(2)}g-\mathrm{P}_{k}^{(2)}g\big) \\
& - \sum_{k=0}^{n-1} \big(\mathbb{E}(f|\mathcal{F}_{k+1})-\mathbb{E}(f|\mathcal{F}_k)\big)
\,\big(\mathrm{P}_{k+1}^{(2)}g-\mathbb{E}(g|\mathcal{G}_{k+1})\big) \\
& + \mathbb{E}(f|\mathcal{F}_n) \big(\mathrm{P}_{n}^{(2)}g-\mathbb{E}(g|\mathcal{G}_n)\big)
- \mathbb{E}(f|\mathcal{F}_0) \big(\mathrm{P}_{0}^{(2)}g-\mathbb{E}(g|\mathcal{G}_0)\big) ,
\end{align*}
so that by the Cauchy-Schwarz inequality
\begin{align*}
\big|T_{n}(f,g)-\widetilde{T}_{n}(f,g)\big| \leq\,
& \big(\mathcal{S}_{\textup{JSW}}^{(1)}f\big) \big(\mathcal{S}_{\textup{conv}}^{(2)}g\big)
+ \big(\mathcal{S}_{\textup{mart}}^{(1)}f\big) \big(\mathcal{S}_{\textup{JSW}}^{(2)}g\big) \\
& + |\mathbb{E}(f|\mathcal{F}_n)| \big(|\mathrm{P}_{n}^{(2)}g|+|\mathbb{E}(g|\mathcal{G}_n)|\big)
+ |\mathbb{E}(f|\mathcal{F}_0)| \big(|\mathrm{P}_{0}^{(2)}g|+|\mathbb{E}(g|\mathcal{G}_0)|\big) ,
\end{align*}
where we have also denoted
$$ \mathcal{S}_{\textup{conv}}^{(2)}g
:= \Big(\sum_{k\in\mathbb{Z}} \big|\mathrm{P}_{k+1}^{(2)}g - \mathrm{P}_{k}^{(2)}g\big|^2\Big)^{1/2},
\quad \mathcal{S}_{\textup{mart}}^{(1)}f
:= \Big(\sum_{k\in\mathbb{Z}} \big|\mathbb{E}(f|\mathcal{F}_{k+1}) - \mathbb{E}(f|\mathcal{F}_{k})\big|^2\Big)^{1/2} . $$
The ``smooth'' square function $\mathcal{S}_{\textup{conv}}^{(2)}$ is known to be bounded (see \cite{DRF}),
while $\mathcal{S}_{\textup{mart}}^{(1)}$ is bounded by the Burkholder-Davis-Gundy inequality \cite{BDG},
\begin{equation}\label{eqstandardsqfn}
\|\mathcal{S}_{\textup{conv}}^{(2)}g\|_{\mathrm{L}^q(\mathbb{R}^d)} \leq C'_q \|g\|_{\mathrm{L}^q(\mathbb{R}^d)},
\quad \|\mathcal{S}_{\textup{mart}}^{(1)}f\|_{\mathrm{L}^p(\mathbb{R}^d)} \leq C'_p \|f\|_{\mathrm{L}^p(\mathbb{R}^d)} .
\end{equation}
Taking $\mathrm{L}^r$-norms we finally get
\begin{align*}
\big\|T_{n}(f,g)-\widetilde{T}_{n}(f,g)\big\|_{\mathrm{L}^{r}(\mathbb{R}^{d})} \leq\,
& \big\|\mathcal{S}_{\textup{JSW}}^{(1)}f\big\|_{\mathrm{L}^{p}(\mathbb{R}^{d})}
\big\|\mathcal{S}_{\textup{conv}}^{(2)}g\big\|_{\mathrm{L}^{q}(\mathbb{R}^{d})} \\
& + \big\|\mathcal{S}_{\textup{mart}}^{(1)}f\big\|_{\mathrm{L}^{p}(\mathbb{R}^{d})}
\big\|\mathcal{S}_{\textup{JSW}}^{(2)}g\big\|_{\mathrm{L}^{q}(\mathbb{R}^{d})} \\
& + 4 \|f\|_{\mathrm{L}^{p}(\mathbb{R}^{d})} \|g\|_{\mathrm{L}^{q}(\mathbb{R}^{d})}
\end{align*}
and it remains to use \eqref{eqjswineq} and \eqref{eqstandardsqfn}.

Let us remark that the particular case of the result in \cite{Kov2} did not really need the full generality of
the result from \cite{JSW}, as the convolution-type paraproduct was compared with its martingale variant
with respect to the standard one-dimensional dyadic grids only.

\section{Proof of Corollaries 3 and 4}
\label{secfoundation}

Let $(\mathcal{F}_s)_{s\geq 0}$, $(\mathcal{G}_s)_{s\geq 0}$, $(X_s)_{s\geq 0}$, $(Y_s)_{s\geq 0}$, $(H_s)_{s\geq 0}$
be as in Subsection \ref{subsecintegral} and fix $t>0$.
Both of our results are immediate consequences of discrete-time Estimate \eqref{eqdiscreteestimate}.
We find it convenient to prove them simultaneously.

Take an elementary predictable integrand \eqref{eqelementaryprocess} associated with some partition \eqref{eqpartition} of $[0,t]$.
We can apply \eqref{eqdiscreteestimate} to the processes $X=(X_{t_k})_{k=0}^{n}$, $Y=(Y_{t_k})_{k=0}^{n}$, and $K=(K_k)_{k=0}^{n-1}$.
Observing that by the discrete-time It\={o} isometry
$$ \|H\|_{X,Y,t}^{2} = \|(K\cdot X)_n\|_{\mathrm{L}^2}^{2} + \|(K\cdot Y)_n\|_{\mathrm{L}^2}^{2} , $$
we establish \eqref{eqcontest}. The second claim of Corollary~\ref{corito} then follows from \eqref{eqcontest} and linearity.

If we also assume $\|K_k\|_{\mathrm{L}^\infty}\leq 1$ for $k=0,1,\ldots,n-1$, then from Definition \eqref{eqseminorm} we have
\begin{align*}
\|H\|_{X,Y,t}^{2}\, & \leq\, \mathbb{E}\,\sum_{k=1}^{n}(X_{t_k}\!-\!X_{t_{k-1}})^2
+\mathbb{E}\,\sum_{k=1}^{n}(Y_{t_k}\!-\!Y_{t_{k-1}})^2 \\
& =\, \mathbb{E}\,\big(X_{t}^2-X_{0}^2+Y_{t}^2-Y_{0}^2\big)
\leq \|X_t\|_{\mathrm{L}^2}^{2} + \|Y_t\|_{\mathrm{L}^2}^{2} .
\end{align*}
Using \eqref{eqcontest} we obtain the following bound on the $\mathrm{L}^{4/3}$-norm of \eqref{eqstochint},
\begin{equation}\label{eqbdauxest}
\Big\|\int_{0}^{t}H_s d(X_s Y_s)\Big\|_{\mathrm{L}^{4/3}} \leq
C \,\big(\|X_t\|_{\mathrm{L}^2}^{2} + \|Y_t\|_{\mathrm{L}^2}^{2}\big)^{1/2}
\big(\|X_t\|_{\mathrm{L}^4}+\|Y_t\|_{\mathrm{L}^4}\big) .
\end{equation}
The right hand side of \eqref{eqbdauxest} is finite because we have assumed $X_t,Y_t\in\mathrm{L}^{4}$ and the proof is complete.

\section{Additional remarks}
\label{secclosing}

We want to prove that $\sigma$-algebras $\mathcal{F}_k$ and $\mathcal{G}_\ell$ defined in the introduction
are independent conditionally on their intersection for any times $k$ and $\ell$.
Indeed, this is a consequence of the more general fact that their conditional expectations commute,
as we show in the next proposition.

\begin{proposition}\label{propcommuting}
The following statements are equivalent for $\sigma$-algebras $\mathcal{F}$ and $\mathcal{G}$ on the same probability space.
\begin{itemize}
\item[(a)]
$\mathbb{E}\big(\mathbb{E}(X|\mathcal{F})|\mathcal{G}\big) = \mathbb{E}\big(\mathbb{E}(X|\mathcal{G})|\mathcal{F}\big)$
for each $X\in\mathrm{L}^1$.
\item[(b)]
$\mathcal{F}$ and $\mathcal{G}$ are independent conditionally on $\mathcal{F}\cap\mathcal{G}$, i.e.
$$ \mathbb{P}(A\cap B|\mathcal{F}\cap\mathcal{G}) = \mathbb{P}(A|\mathcal{F}\cap\mathcal{G}) \,\mathbb{P}(B|\mathcal{F}\cap\mathcal{G}) $$
whenever $A\in\mathcal{F}$ and $B\in\mathcal{G}$.
\end{itemize}
\end{proposition}

\begin{proof}
(a)$\Rightarrow$(b):
We noticed in the proof of Lemma~\ref{lemmaformulae} that Condition (a) also implies
$$ \mathbb{E}\big(\mathbb{E}(X|\mathcal{F})|\mathcal{G}\big) = \mathbb{E}\big(\mathbb{E}(X|\mathcal{G})|\mathcal{F}\big)
= \mathbb{E}(X|\mathcal{F}\cap\mathcal{G}) . $$
Applying this to the indicator function of any $B\in\mathcal{G}$ gives
$$ \mathbb{E}(\mathbf{1}_{B}|\mathcal{F})
= \mathbb{E}\big(\mathbb{E}(\mathbf{1}_{B}|\mathcal{G})|\mathcal{F}\big)
= \mathbb{E}(\mathbf{1}_{B}|\mathcal{F}\cap\mathcal{G}) $$
and thus for $A\in\mathcal{F}$ we have
\begin{align*}
\mathbb{E}(\mathbf{1}_{A}\mathbf{1}_{B}|\mathcal{F}\cap\mathcal{G})
& = \mathbb{E}\big(\mathbb{E}(\mathbf{1}_{A}\mathbf{1}_{B}|\mathcal{F})\big|\mathcal{F}\cap\mathcal{G}\big)
= \mathbb{E}\big(\mathbf{1}_{A}\mathbb{E}(\mathbf{1}_{B}|\mathcal{F})\big|\mathcal{F}\cap\mathcal{G}\big)\\
& = \mathbb{E}\big(\mathbf{1}_{A}\mathbb{E}(\mathbf{1}_{B}|\mathcal{F}\cap\mathcal{G})\big|\mathcal{F}\cap\mathcal{G}\big)
= \mathbb{E}(\mathbf{1}_{A}|\mathcal{F}\cap\mathcal{G}) \,\mathbb{E}(\mathbf{1}_{B}|\mathcal{F}\cap\mathcal{G}) .
\end{align*}

(b)$\Rightarrow$(a): Take $X\in\mathrm{L}^1$. By symmetry it is enough to show
$\mathbb{E}\big(\mathbb{E}(X|\mathcal{F})|\mathcal{G}\big) = \mathbb{E}(X|\mathcal{F}\cap\mathcal{G})$.
The independence condition applied to $\mathbb{E}(X|\mathcal{F})$ and $B\in\mathcal{G}$ yields
\begin{align*}
\mathbb{E}\big(\mathbb{E}(X|\mathcal{F})\mathbf{1}_{B}\big|\mathcal{F}\cap\mathcal{G}\big)
& = \mathbb{E}(X|\mathcal{F}\cap\mathcal{G}) \,\mathbb{E}(\mathbf{1}_{B}|\mathcal{F}\cap\mathcal{G})\\
& = \mathbb{E}\big(\mathbb{E}(X|\mathcal{F}\cap\mathcal{G})\mathbf{1}_{B}\big|\mathcal{F}\cap\mathcal{G}\big) .
\end{align*}
Taking expectation gives
$$ \mathbb{E}\big(\mathbb{E}(X|\mathcal{F})\mathbf{1}_{B}\big)
= \mathbb{E}\big(\mathbb{E}(X|\mathcal{F}\cap\mathcal{G})\mathbf{1}_{B}\big) , $$
which proves the claim.
\end{proof}

Recall that Condition (a) was verified for $\mathcal{F}_k$ and $\mathcal{G}_\ell$ in Lemma~\ref{lemmaformulae};
see Equation \eqref{eqexpcommute}.

\smallskip
Now we turn to a simple example showing that the construction described in Subsection \ref{subsecintegral} cannot be
realized along the lines of the classical one.
Suppose that $(A_t)_{t\geq 0}$ and $(B_t)_{t\geq 0}$ are standard one-dimensional Brownian motions constructed on
$(\Omega_1,\mathcal{A},\mathbb{P}_1)$ and $(\Omega_2,\mathcal{B},\mathbb{P}_2)$ respectively.
Let the two filtrations on the product space be given by \eqref{eqfiltrations2}, so that
$$ \mathcal{F}_t\cap\mathcal{G}_t = \sigma\big( \big\{A_s,B_s : 0\leq s\leq t \big\}\big) $$
is just the natural filtration of a two-dimensional Brownian motion.
Define
\begin{equation}\label{eqexample}
X_t(\omega_1,\omega_2) := A_t(\omega_1)V(\omega_2), \quad Y_t(\omega_1,\omega_2) := U(\omega_1)B_t(\omega_2),
\end{equation}
for some $\mathcal{A}$-measurable random variable $U$ and some $\mathcal{B}$-measurable random variable $V$.
At any time $t=\varepsilon>0$ we have
$$ (X_\varepsilon Y_\varepsilon)(\omega_1,\omega_2) = (A_\varepsilon U)(\omega_1) (B_\varepsilon V)(\omega_2) . $$
Since $\mathbb{P}_{1}(A_\varepsilon=0)=\mathbb{P}_{2}(B_\varepsilon=0)=0$, we see that $X_\varepsilon Y_\varepsilon$
can be equal a.s.\@ to an indicator function of any $C\times D\in\mathcal{A}\otimes\mathcal{B}$.
This in turn implies that (the completion of) the $\sigma$-algebra generated by such products $X_\varepsilon Y_\varepsilon$
can contain the whole product $\sigma$-algebra for any fixed time $\varepsilon>0$.
We see that the process $(X_t Y_t)_{t\geq 0}$ need not be adapted with respect to $(\mathcal{F}_t\cap\mathcal{G}_t)_{t\geq 0}$.
On the other extreme, the paths of $(X_t Y_t)_{t\geq 0}$ do not necessarily have bounded variation,
which is seen by an even simpler counterexample
$$ X_t(\omega_1,\omega_2) := A_t(\omega_1), \quad Y_t(\omega_1,\omega_2) := 1 . $$
Example \eqref{eqexample} also rules out the possibility of a canonical way of decomposing $(X_t Y_t)_{t\geq 0}$
into a sum of a finite variation process and a local martingale.
For instance, a choice that works for continuous martingales with respect to the same Brownian filtration
\begin{align*}
& X_t Y_t = \langle X_t,Y_t\rangle + M_t, \\
& \langle X_t,Y_t\rangle:=\lim_{m\to\infty}\sum_{k=1}^{m}
\big(X_{\frac{kt}{m}}-X_{\frac{(k-1)t}{m}}\big)\big(Y_{\frac{kt}{m}}-Y_{\frac{(k-1)t}{m}}\big)
\quad\textup{in probability}
\end{align*}
does not necessarily give a process $(M_t)_{t\geq 0}$ adapted to $(\mathcal{F}_t\cap\mathcal{G}_t)_{t\geq 0}$.

\section*{Acknowledgments}
We would like to thank Professor Zoran Vondra\v{c}ek for suggestions on how to clarify the presented material
and Professor Hrvoje \v{S}iki\'{c} for a useful discussion.

\end{document}